\documentclass[11pt]{article}
\usepackage{mathrsfs}
\usepackage{amsmath}
\usepackage{color}
\usepackage{amsmath} \allowdisplaybreaks
\usepackage{amsmath,amsthm,amssymb,amscd}
\usepackage{latexsym}
\usepackage{hyperref}
\usepackage[numbers,sort&compress]{natbib}
\usepackage{hypernat}

\newtheorem{theorem}{Theorem}[section]

\newtheorem{lemma}[theorem]{Lemma}

\newtheorem{definition}[theorem]{Definition}
\newtheorem{remark}[theorem]{Remark}

\theoremstyle{definition} \theoremstyle{remark}
\numberwithin{equation}{section}

\begin{document}

\title{{\bf Sharper $L^1$ convergence rates to the Barenblatt profile for the compressible Euler equations with time-dependent damping}
\thanks{The work was supported partly by the NSF of China (12171094) and the Shanghai Key
Laboratory for Contemporary Applied Mathematics (08DZ2271900).}}
\author{Jun-Ren Luo$^{a}$\thanks{Corresponding author. E-mail: \texttt{junrenluo@usst.edu.cn}(J. Luo),\texttt{tjxiao@fudan.edu.cn}(T.Xiao)},  Ti-Jun Xiao$^{b}$
\\{\small $^a$  College of Science, University of Shanghai for Science and Technology,}\\{\small Shanghai 200093, China}\\
{\small $^b$  Shanghai Key
Laboratory for Contemporary Applied Mathematics}\\ {\small School of Mathematical Sciences, Fudan University, Shanghai 200433, China}}

\date{}
\maketitle

\begin{minipage}[t]{14cm}

\noindent {\bf Abstract:}
We study the large-time behavior of weak entropy solutions with bounded density and velocity to the one-dimensional compressible Euler equations with time-dependent damping and vacuum. Assuming that the spatial primitive of the initial density perturbation belongs to $L^2(\mathbf{R})$, we establish convergence rates to the Barenblatt solution with the same finite mass in $L^1$ for all $\gamma>1$ and $0\le\nu<1$, and in $L^\gamma$ for $\gamma\ge2$, where $\gamma$ is the adiabatic exponent and $\nu$ is the damping exponent.
On the common parameter ranges, the leading $L^1$ decay exponent exceeds the exponents in the earlier estimates compared below, with equality only at $(\gamma,\nu)=(2,0)$. The proof combines a structural inequality for the higher-order entropy remainder, valid for every $\gamma>1$, with estimates for the relative pressure and time-weighted entropy inequalities.
The resulting decay exponent has an explicit formula in each damping regime.

\vspace{0.4cm}



\end{minipage}

\section{Introduction}
Our starting point is the Cauchy problem for a compressible gas on the real line. In Eulerian variables, the density and momentum obey
\begin{eqnarray}\label{1.1}
\left\{\begin{array}{ll}
\rho_t+(\rho u)_x=0, \\
(\rho u)_t+\left(\rho u^2+p(\rho)\right)_x=-\displaystyle\frac{\alpha}{(1+t)^{\nu}}\rho u,
\end{array} \right.
\end{eqnarray}
with the initial data
\begin{equation}\label{r0u0}
(\rho, u)(x, 0)=\left(\rho_0, u_0\right)(x), \quad \rho_0(x) \geq 0, \quad \int^{+\infty}_{-\infty} \rho_0(x) d x=M>0,
\end{equation}
where $\rho$, $u$, and $p$ denote the density, velocity, and pressure, respectively. For a polytropic perfect gas, the equation of state is
$$p(\rho)=\kappa \rho^\gamma,$$
with $\kappa=\frac{(\gamma-1)^2}{4\gamma}$, and $\gamma>1$ is the adiabatic exponent. The damping coefficient $\alpha(1+t)^{-\nu}$ is constant for $\nu=0$ and decays in time for $0<\nu<1$. Throughout this paper, we take $$\alpha=\kappa=\frac{(\gamma-1)^2}{4\gamma}.$$
With this choice, the ratio $\kappa/\alpha$ equals one, and the associated diffusion equation takes the form given below.
We write $m:=\rho u$ for the momentum.

For constant friction ($\nu=0$), balancing pressure against damping in the momentum equation leads to the diffusion model
\begin{eqnarray}\label{1.3'}
\left\{\begin{array}{ll}
\bar{\rho}_t=\left(\bar{\rho}^\gamma\right)_{x x}, \\
\bar{m}=-\left(\bar{\rho}^\gamma\right)_x.
\end{array} \right.
\end{eqnarray}
This formal approximation reflects the faster decay of the inertial terms in the momentum equation under the damping effect; see Hsiao and Liu \cite{HSI}. Nishihara \cite{NIS1} subsequently improved the convergence rates to nonlinear diffusion waves in the small-perturbation setting by using detailed energy estimates. For further results on smooth solutions and nonlinear diffusion waves, see \cite{SID,ZHA,JIA}. Using the vanishing viscosity method, Zhu \cite{ZHU1} studied the asymptotic behavior of weak entropy solutions. For the behavior of solutions near vacuum, we refer to \cite{LIU1}. For finite-mass solutions with vacuum, Huang, Marcati and Pan \cite{HUA2} established convergence rates in $L^2$ or $L^\gamma$ to the corresponding Barenblatt solution for every $\gamma>1$. Later, Huang, Pan and Wang \cite{HUA4} obtained the following $L^1$ convergence estimate for $1<\gamma<3$:

$$
\|(\rho-\bar{\rho})(\cdot, t)\|_{L^1} \leq C(1+t)^{-\frac{1}{4(\gamma+1)}+\varepsilon}.
$$
For $\gamma\ge2$, Geng and Huang \cite{GEN1} sharpened this estimate to
$$
\|(\rho-\bar{\rho})(\cdot, t)\|_{L^1} \leq C(1+t)^{-\frac{1}{(\gamma+1)^2}+\varepsilon},
$$
in the same finite-mass setting. For Euler equations with friction, see \cite{CAR}, and for compressible Euler equations with physical vacuum free boundaries, see \cite{ZEN,ZEN1}. A related relative-entropy approach in parabolic scaling variables was developed by Eiter and Schindler \cite{EIT} for constant damping.

When $0<\nu<1$, the damping coefficient in \eqref{1.1} decays in time. Darcy's law suggests that finite-mass solutions approach the Barenblatt solution associated with the following system:
\begin{eqnarray}\label{1.3}
\left\{\begin{array}{ll}
\bar{\rho}_t=(1+t)^{\nu}\left(\bar{\rho}^\gamma\right)_{x x}, \\
\bar{m}=-(1+t)^{\nu}\left(\bar{\rho}^\gamma\right)_x,
\end{array} \right.
\end{eqnarray}
where $\eqref{1.3}_1$ is the PME with time-dependent diffusion and $\eqref{1.3}_2$ is Darcy's law. For compressible Euler equations with time-dependent damping, see \cite{CUI, HOU1, SUG1}. Pan \cite{PAN1,PAN2} identified the critical damping regime $\nu=1$ and the threshold $\alpha=2$ separating global existence from finite-time blow-up for small-data $C^1$ solutions in one dimension. For further results on blow-up, we refer to \cite{HOU, SUG}. Also, Cui-Yin-Zhang-Zhu \cite{CUI1} and Li-Li-Mei-Zhang \cite{LI} independently studied the convergence of smooth solutions with $0\le\nu<1$ and $\alpha>0$.

In \cite{GEN3}, for $1<\gamma<3$, the authors proved that any $L^{\infty}$ weak entropy solution of problem \eqref{1.1}-\eqref{r0u0} converges to the Barenblatt solution of equation \eqref{1.3} with the same mass in $L^1$ at the rate
$$
||(\rho-\bar{\rho})(\cdot,t)||_{L^{1}}\leq C(1+t)^{-a_2(\varepsilon)},
$$
with
$$
a_2(\varepsilon)=
\left\{\begin{aligned}
&\frac{\nu+1}{4(\gamma+1)}-\varepsilon,\quad\nu\in\left(0,\frac{\gamma}{\gamma+2}\right],\\
&\frac{1-\nu}{4}-\varepsilon,\quad\nu\in\left[\frac{\gamma}{\gamma+2},1\right).
\end{aligned}\right.
$$
Subsequently, Cui, Yin and Zhu \cite{CUI2} obtained
$$
||(\rho-\bar{\rho})(\cdot,t)||_{L^{1}}\leq C(1+t)^{-k_1(\epsilon)},
$$
with
$$
k_1(\epsilon)=\left\{\begin{array}{lll}
\displaystyle\frac{\gamma^2-\gamma-1-\nu\left(\gamma^2+\gamma-1\right)}{(\gamma+1)(2 \gamma-1)}-\epsilon, & \displaystyle\frac{1+\sqrt{5}}{2}<\gamma \leq 2, & 0 \leq \nu<\displaystyle\frac{\gamma^2-\gamma-1}{\gamma^2+\gamma-1}, \\
\displaystyle\frac{1-\nu(2 \gamma+1)}{(\gamma+1)^2}-\epsilon, & \gamma \geq 2, & 0 \leq \nu<\displaystyle\frac{1}{2 \gamma+1} .
\end{array}\right.
$$

In this paper, we establish $L^1$ and $L^\gamma$ convergence rates for $L^\infty$ weak entropy solutions of \eqref{1.1} to the Barenblatt solution of \eqref{1.3} with the same finite mass. The $L^1$ estimates hold for every $\gamma>1$ and $0\le\nu<1$, while the $L^\gamma$ estimates hold for $\gamma\ge2$ and $0\le\nu<1$. Our argument combines the structural inequality for the higher-order entropy remainder in Lemma~\ref{lem42} with the relative-pressure bounds in Lemmas~\ref{lem31} and~\ref{lem33}, together with time-weighted entropy estimates. The leading $L^1$ decay exponent is strictly larger than those in \cite{CUI2,GEN3} on their common parameter ranges, except at $(\gamma,\nu)=(2,0)$ in the comparison with \cite{CUI2}; see Remark~\ref{remark:comparison}. The estimates also cover parameter ranges not included in those references. For constant damping, the resulting estimate recovers the rate in \cite{GEN1} when $\gamma=2$ and improves the cited rates in the other common cases; see Remark~\ref{remark:constant-comparison}.

The paper is organized as follows. In Section~2, we recall the Barenblatt profile, state the main result, and outline the proof. Section~3 collects the auxiliary lemmas and a priori estimates. Section~4 addresses the restriction on $\gamma$ in the weighted entropy estimates of \cite{GEN3}. Those estimates use the bound $m\partial_m A\ge 3A\ge0$ for the higher-order entropy remainder. For $\gamma>3$, the exponent $\sigma=2/(\gamma-1)$ falls below one, so convexity of the power kernel $|b+z|^\sigma$ can no longer be used to deduce monotonicity after integration in $z$. Lemma~\ref{lem42} establishes this monotonicity directly in that range and yields $m\partial_m D\ge2D\ge0$ for the remainder $D$ for every $\gamma>1$. The coefficient $2$ suffices because, for $0\le\nu<1$, the term produced by differentiating the time weight carries an extra factor $(1+t)^{\nu-1}$ relative to the damping dissipation and can be absorbed for large times. This closes the weighted estimates throughout that parameter range. Section~5 combines them with the relative-pressure bounds to obtain the $L^1$ decay rate for every $\gamma>1$ and $0\le\nu<1$, as well as the $L^\gamma$ estimate for $\gamma\ge2$.

\section{Preliminaries and main result }
Throughout this paper, $\|\cdot\|_{L^p}$ denotes the norm in $L^p(\mathbf{R})$, $1\le p\le\infty$, and $\|\cdot\|:=\|\cdot\|_{L^2}$. Unless a domain is specified, spatial integrals are taken over $\mathbf{R}$:
$$\int f(x,\cdot)\,dx:=\int_{-\infty}^{+\infty}f(x,\cdot)\,dx.$$
The symbol $C$ denotes a positive constant independent of time; its value may change from line to line.
Vacuum is allowed throughout. We therefore formulate the Cauchy problem through distributional conservation laws and convex entropy inequalities, using the weak-solution framework of \cite{CUI2,GEN3}.

\begin{definition}
{\rm A pair $(\rho,m)\in L^\infty(\mathbf{R}\times(0,\infty))$, with $\rho\ge0$, is called a weak entropy solution of \eqref{1.1}--\eqref{r0u0} if, for every test function $\phi\in C_c^\infty(\mathbf{R}\times[0,\infty))$, it satisfies
$$
\left\{\begin{array}{ll}
\displaystyle\int_0^{+\infty} \int\left(\rho \phi_t+m \phi_x\right) d x d t+\int_{-\infty}^{+\infty} \rho_0(x) \phi(x, 0) d x=0, \\
\displaystyle\int_0^{+\infty} \int\left\{m \phi_t+\left(\frac{m^2}{\rho}+\kappa\rho^\gamma\right) \phi_x-\frac{\alpha m}{(1+t)^\nu} \phi\right\} d x d t+\int_{-\infty}^{+\infty} m_0(x) \phi(x, 0) d x=0,
\end{array} \right.
$$
where $m_0=\rho_0u_0$. For every nonnegative $\phi\in C_c^\infty(\mathbf R\times[0,\infty))$, we also require
\[
\int_0^\infty\!\!\int
 \left(\eta\,\phi_t+q\,\phi_x-\frac{\alpha m\eta_m}{(1+t)^\nu}\phi\right)\,dx\,dt
 +\int\eta(\rho_0,m_0)\phi(x,0)\,dx\ge0
\]
for every weak convex entropy-flux pair $(\eta,q)$ satisfying
$$
\nabla q=\nabla \eta \nabla f, \quad f=\left(m, \frac{m^2}{\rho}+\kappa \rho^\gamma\right)^t, \quad \eta(0,0)=0 .
$$}
\end{definition}

The kinetic representation of Lions, Perthame and Tadmor \cite{LIO} provides the entropy pairs used in our estimates. For a generating function $g$, the representation reads
\begin{equation}\label{etaq}
\begin{aligned}
\eta(\rho, m) & =\rho \int_{-1}^1 g\left(u+z \rho^\theta\right)\left(1-z^2\right)^\lambda \mathrm{d} z, \\
q(\rho, m) & =\rho \int_{-1}^1 g\left(u+z \rho^\theta\right)\left(u+\theta z \rho^\theta\right)\left(1-z^2\right)^\lambda \mathrm{d} z,
\end{aligned}
\end{equation}
where $m=\rho u, \theta=\frac{\gamma-1}{2}, \lambda=\frac{3-\gamma}{2(\gamma-1)}$ and the generators used here are convex $C^2$ functions.
The energy and higher-order estimates below use, respectively, $g(\xi)=|\xi|^2/(2J_\lambda)$, with $J_\lambda=\int_{-1}^1(1-z^2)^\lambda\,dz$, and $g(\xi)=|\xi|^{\frac{2\gamma}{\gamma-1}}$.

The comparison density is the mass-$M$ self-similar profile for the diffusion problem
\begin{eqnarray}\label{3.1}
\left\{\begin{array}{ll}
\bar{\rho}_t=(1+t)^{\nu}\left(\bar{\rho}^\gamma\right)_{x x}, \\
\bar{\rho}(x,-1)=M \delta(x), \quad M>0.
\end{array} \right.
\end{eqnarray}
For $0\le\nu<1$ and $\gamma>1$, we have
\begin{equation}\label{barrho}
\bar{\rho}(x, t)=(1+t)^{-\frac{1+\nu}{\gamma+1}}\left(A_0-B_0 x^2 (1+t)^{-\frac{2(1+\nu)}{\gamma+1}}\right)_{+}^{\frac{1}{\gamma-1}},
\end{equation}
where $(f)_{+}=\max \{0, f\}$,
$$
B_0=\frac{(\gamma-1)(1+\nu)}{2\gamma(\gamma+1)}, \quad \text { and } \quad A_{0}^{\frac{\gamma+1}{2(\gamma-1)}}=M \sqrt{B_0}\left(\int_{-1}^1\left(1-y^2\right)^{1 /(\gamma-1)} d y\right)^{-1} \text {. }
$$
As a result, the velocity is given by
\begin{equation}\label{ux}
\bar{u}(x, t)=\frac{1+\nu}{(\gamma+1)(1+t)}x,\,|x| < \sqrt{\frac{A_0}{B_0}}(1+t)^{\frac{1+\nu}{\gamma+1}}.
\end{equation}

With this comparison profile fixed, the decay estimates can be stated as follows.

\begin{theorem}\label{thm1}
Take an initial density $\rho_0\in L^1(\mathbf{R})\cap L^\infty(\mathbf{R})$ and an initial velocity $u_0\in L^\infty(\mathbf{R})$, subject to
$$
M=\int_{-\infty}^{+\infty} \rho_0(x) d x>0, \quad \rho_0(x) \geq 0 .
$$
Fix $\gamma>1$ and $\nu\in[0,1)$. Consider a bounded entropy solution $(\rho,m)$ of \eqref{1.1}--\eqref{r0u0} satisfying
$$
0 \leq \rho(x, t) \leq C, \quad|m(x, t)| \leq C \rho(x, t),
$$
where $m^2/\rho$ and $m\eta_m$ are assigned their vacuum value zero at $(0,0)$. Consider the mass-matched profile $\bar\rho$ in \eqref{barrho}, with Darcy momentum $\bar m=-(1+t)^\nu(\bar\rho^\gamma)_x$. Set
$$
y(x, t)=-\int_{-\infty}^x(\rho-\bar{\rho})(r, t) d r,
$$
and let $y(x, 0) \in L^2(\mathbf{R})$. If $2\le\gamma<+\infty$, then for any $\varepsilon>0$ and $t>0$,
\begin{equation}\label{1rr+1}
\|(\rho-\bar{\rho})(\cdot, t)\|_{L^\gamma}^\gamma \leq C_\varepsilon(1+t)^{-\mu(\varepsilon)},
\end{equation}
where
\begin{eqnarray}\label{1.5}
\mu(\varepsilon):=\left\{\begin{array}{ll}
\displaystyle\frac{\gamma^2+\gamma-1}{(\gamma+1)^2}(1+\nu)- \varepsilon,\quad&\nu\in\left[0,\displaystyle\frac{\gamma}{\gamma+2}\right], \\[0.36cm]
\displaystyle\frac{2\gamma-1-\nu}{\gamma+1}- \varepsilon,\quad&\nu\in\left[\displaystyle\displaystyle\frac{\gamma}{\gamma+2},1\right).
\end{array} \right.
\end{eqnarray}
Furthermore, for any $\gamma>1$, it holds that
\begin{equation}\label{1rr+1'}
\begin{aligned}
 \|(\rho-\bar{\rho})(\cdot, t)\|_{L^1} \leq C_\varepsilon(1+t)^{-k(\varepsilon)} ,
\end{aligned}
\end{equation}
with
\begin{eqnarray}\label{1.6}
k(\varepsilon):=\left\{\begin{array}{ll}
\displaystyle\frac{\gamma}{2(\gamma+1)^2}(1+\nu)-\varepsilon,\quad&\nu\in\left[0,\displaystyle\frac{\gamma}{\gamma+2}\right], \\[0.36cm]
\displaystyle\frac{\gamma}{2(\gamma+1)}(1-\nu)-\varepsilon,\quad&\nu\in\left[\displaystyle\frac{\gamma}{\gamma+2},1\right).
\end{array} \right.
\end{eqnarray}

\end{theorem}

\begin{remark}\label{remark:comparison}{\rm
We compare the leading decay exponents, before the arbitrarily small
$\varepsilon$-loss. Let $K(\gamma,\nu)=k(\varepsilon)+\varepsilon$,
where $k(\varepsilon)$ is defined in \eqref{1.6}.
The corresponding exponent in \cite{CUI2} is
\[
K_{1}(\gamma,\nu)=
\begin{cases}
\displaystyle
\frac{\gamma^2-\gamma-1-\nu(\gamma^2+\gamma-1)}
     {(\gamma+1)(2\gamma-1)},
&\displaystyle
\frac{1+\sqrt5}{2}<\gamma\le2,\quad
0\le\nu<\frac{\gamma^2-\gamma-1}{\gamma^2+\gamma-1},\\[9pt]
\displaystyle\frac{1-\nu(2\gamma+1)}{(\gamma+1)^2},
&\displaystyle \gamma\ge2,\quad 0\le\nu<\frac1{2\gamma+1}.
\end{cases}
\]
Both admissible ranges lie in $0\le\nu<\gamma/(\gamma+2)$, since
\[
\frac{\gamma^2-\gamma-1}{\gamma^2+\gamma-1}
<\frac\gamma{\gamma+2},
\qquad
\frac1{2\gamma+1}<\frac\gamma{\gamma+2}
\]
on their respective ranges. Consequently, the first branch of
\eqref{1.6} applies. If $(1+\sqrt5)/2<\gamma\le2$, direct computation
gives
\[
\begin{aligned}
K(\gamma,\nu)-K_{1}(\gamma,\nu)=
\frac{(2-\gamma)(2\gamma^2+2\gamma+1)
      +\nu(2\gamma^3+6\gamma^2-\gamma-2)}
     {2(\gamma+1)^2(2\gamma-1)}\ge0.
\end{aligned}
\]
If $\gamma\ge2$, then
\[
K(\gamma,\nu)-K_{1}(\gamma,\nu)
=\frac{\gamma-2+\nu(5\gamma+2)}{2(\gamma+1)^2}\ge0.
\]
In both cases equality occurs only at $(\gamma,\nu)=(2,0)$.

For $1<\gamma<3$ and $0<\nu<1$, the leading exponent obtained in
\cite{GEN3} is
\[
K_{2}(\gamma,\nu)=
\begin{cases}
\displaystyle\frac{1+\nu}{4(\gamma+1)},
&\displaystyle 0<\nu\le\frac\gamma{\gamma+2},\\[7pt]
\displaystyle\frac{1-\nu}{4},
&\displaystyle \frac\gamma{\gamma+2}\le\nu<1.
\end{cases}
\]
On this entire range,
\[
K(\gamma,\nu)
=\frac{2\gamma}{\gamma+1}K_{2}(\gamma,\nu)
>K_{2}(\gamma,\nu).
\]
Thus Theorem~\ref{thm1} improves the leading $L^1$ decay exponent on the
common parameter ranges, with the sole equality case
$(\gamma,\nu)=(2,0)$ in the comparison with \cite{CUI2}.
The theorem also provides a single formula for each damping regime
throughout $(\gamma,\nu)\in(1,\infty)\times[0,1)$.}
\end{remark}

\begin{remark}\label{remark:constant-comparison}
{\rm When $\nu=0$, it is shown in \cite{GEN1} for $\gamma\ge2$ and \cite{HUA4} for $1<\gamma<3$ that $$\|(\rho-\bar{\rho})(\cdot, t)\|_{L^1}\leq C_{\varepsilon}(1+t)^{-\frac{1}{(\gamma+1)^2}+\varepsilon}$$ and $$\|(\rho-\bar{\rho})(\cdot, t)\|_{L^1}\leq C_{\varepsilon}(1+t)^{-\frac{1}{4(\gamma+1)}+\varepsilon}$$ respectively. It is easy to check that $$\frac{\gamma}{2(\gamma+1)^2}>\frac{1}{4(\gamma+1)} \quad \mbox{for} \ \gamma>1,$$
$$\frac{\gamma}{2(\gamma+1)^2}>\frac{1}{(\gamma+1)^2}  \quad \mbox{for} \ \gamma>2,$$
and $\frac{\gamma}{2(\gamma+1)^2}=\frac{1}{(\gamma+1)^2}$ when $\gamma=2$. Thus, compared with the previous results, Theorem \ref{thm1} provides the following unified estimate, with strict improvement on the common parameter ranges except at $\gamma=2$:
\begin{equation*}
  \|(\rho-\bar{\rho})(\cdot, t)\|_{L^1} \leq C_\varepsilon(1+t)^{-\frac{\gamma}{2(\gamma+1)^2}+\varepsilon}
\end{equation*}
valid for all $\gamma> 1$.}
\end{remark}

We now outline the proof. Following the strategy of \cite{CUI2,GEN3}, we use entropy inequalities with suitable time weights. A key step is to extend the intermediate estimates in \cite[Lemmas~3.8 and~3.9]{GEN3} from $1<\gamma<3$ to all $\gamma>1$; see Theorem~\ref{thm3}. This extension relies on Taylor's formula and the structural inequality for the entropy remainder established in Lemma~\ref{lem42}.

For $\gamma\ge2$, we use the mechanical energy
$$
\eta=\frac{1}{2}\frac{m^2}{\rho}+\frac{\kappa}{\gamma-1}\rho^\gamma.
$$
The mechanical energy controls the $L^\gamma$ density term, whereas the higher-order entropy yields $L^{\gamma+1}$ estimates. To connect these estimates, Lemma~\ref{lem31} bounds the relative pressure
$$
\rho^\gamma-\bar\rho^\gamma-\gamma\bar\rho^{\gamma-1}(\rho-\bar\rho)
$$
in terms of $[(\rho^\gamma-\bar\rho^\gamma)(\rho-\bar\rho)]^{\gamma/(\gamma+1)}$. This bound, together with the time-weighted estimates, yields the relative entropy decay in Theorem~\ref{thm4}.

For the $L^1$ estimate with $\gamma>1$, the case $\gamma\ge2$ can be handled using the lower bound
\[
\rho^\gamma-\bar\rho^\gamma
   -\gamma\bar\rho^{\gamma-1}(\rho-\bar\rho)
\ge c_\gamma\bar\rho^{\gamma-2}|\rho-\bar\rho|^2,
\qquad\text{on }\{\bar\rho>0\}.
\]
When $1<\gamma<2$, the inequalities in Lemma~\ref{lem32} remain valid; however, the uniform lower bound
\[
\rho^\gamma-\bar\rho^\gamma
   -\gamma\bar\rho^{\gamma-1}(\rho-\bar\rho)
\ge c_\gamma|\rho-\bar\rho|^\gamma
\]
fails near $\rho=\bar\rho>0$. We therefore distinguish large and small relative density differences. Lemma~\ref{lem33} provides a lower bound by
$c_1|\rho-\bar\rho|^\gamma$ when $|\rho-\bar\rho|\ge\bar\rho/2$, and by $c_2\bar\rho^{\gamma-2}|\rho-\bar\rho|^2$ when
$|\rho-\bar\rho|<\bar\rho/2$ on $\{\bar\rho>0\}$. Together with the equal-mass reduction in Lemma~\ref{lem31.5}, these two bounds
yield the required $L^1$ estimate.

\section{Some lemmas and a priori estimates}
We collect the auxiliary results used in the proof of the main theorem.

\begin{lemma}\label{lem31}
Let \(1<\gamma<\infty\). Then there exists a constant
\(c=c(\gamma)>0\) such that, for every
\(\rho,\bar{\rho}\geq 0\),
\begin{equation}\label{gam1}
\begin{aligned}
\rho^\gamma-\bar{\rho}^\gamma
-\gamma\bar{\rho}^{\gamma-1}(\rho-\bar{\rho})\leq c\big[(\rho^\gamma-\bar{\rho}^\gamma)(\rho-\bar{\rho})\big]^{\frac\gamma{\gamma+1}}.
\end{aligned}
\end{equation}
\end{lemma}

\begin{proof}
If \(\bar{\rho}=0\), then \eqref{gam1} holds for every \(c\ge 1\). If \(\rho=0\), then \eqref{gam1} holds whenever $c\ge\gamma-1$.

Assume now that \(\rho,\bar{\rho}>0\), and set
\[
s=\frac{\rho}{\bar{\rho}}.
\]
Since $s>0$, homogeneity reduces the assertion to boundedness of
\[
 F(s)=\frac{s^\gamma-1-\gamma(s-1)}
 {[(s^\gamma-1)(s-1)]^{\gamma/(\gamma+1)}},\qquad s\ge0,\quad s\ne1.
\]
Taylor expansion at $s=1$ gives
\[
 F(s)=O(|s-1|^{2/(\gamma+1)}).
\]
Moreover, $F(0)=\gamma-1$ and $\lim_{s\to\infty}F(s)=1$.
Continuity away from $s=1$ completes the proof.
\end{proof}

\begin{lemma}\label{lem31.5}\cite{HUA4}
Let $\rho,\bar\rho\ge0$ have the same total mass $M$, that is,
$$
\int_{-\infty}^{+\infty}\rho\,dx=\int_{-\infty}^{+\infty}\bar\rho\,dx=M.
$$
Then, for every $t>0$, it holds
\begin{equation}\label{rhoneq0}
\int_{-\infty}^{+\infty}|\rho-\bar{\rho}|\mathrm{d} x \leq 2 \int_{\{x:\bar{\rho}>0\}}|\rho-\bar{\rho}| \mathrm{d} x.
\end{equation}
\end{lemma}

\begin{lemma}\label{lem32}\cite{HUA2,HUA4}
Assume $0 \leq \rho, \bar{\rho} \leq C$ and $\gamma>1$. Then there exist two positive constants $d_1$ and $d_2$ such that
$$
\left\{\begin{aligned}
\left(\rho^\gamma-\bar{\rho}^\gamma\right)(\rho-\bar{\rho})&\ge |\rho-\bar{\rho}|^{\gamma+1},\\
d_1\left(\rho^{\gamma-1}+\bar{\rho}^{\gamma-1}\right)(\rho-\bar{\rho})^2 & \leq \rho^{\gamma+1}-\bar{\rho}^{\gamma+1}-(\gamma+1) \bar{\rho}^\gamma(\rho-\bar{\rho}) \\
& \leq d_2\left(\rho^{\gamma-1}+\bar{\rho}^{\gamma-1}\right)(\rho-\bar{\rho})^2, \\
d_1\left(\rho^{\gamma-1}+\bar{\rho}^{\gamma-1}\right)(\rho-\bar{\rho})^2 & \leq\left(\rho^\gamma-\bar{\rho}^\gamma\right)(\rho-\bar{\rho}) \\
& \leq d_2\left(\rho^{\gamma-1}+\bar{\rho}^{\gamma-1}\right)(\rho-\bar{\rho})^2 .
\end{aligned}\right.
$$
\end{lemma}

The uniform lower bound
$$
\rho^\gamma-\bar\rho^\gamma-\gamma\bar\rho^{\gamma-1}(\rho-\bar\rho)
\ge C|\rho-\bar\rho|^\gamma
$$
holds for $\gamma\ge2$, but fails near $\rho=\bar\rho>0$ when $1<\gamma<2$. The following lemma gives the two lower bounds needed to prove \eqref{1rr+1'} in the latter range.

\begin{lemma}\label{lem33}
Let $1<\gamma<2$, and $0\leq \rho,\,\bar{\rho}\leq C$. Then, there exist two positive constants $c_1$ and $c_2$ such that\\
\noindent{(i)} If $\rho=0$, or $\bar{\rho}=0$, or $|\rho-\bar{\rho}|\ge\frac{1}{2}\bar{\rho}$, then
\begin{equation}\label{rho0}
c_1|\rho-\bar{\rho}|^{\gamma}
\leq
\rho^{\gamma}-\bar{\rho}^{\gamma}
-\gamma\bar{\rho}^{\gamma-1}(\rho-\bar{\rho}).
\end{equation}

\noindent{(ii)} If $\rho\neq 0$, $\bar{\rho}\neq 0$ and
$|\rho-\bar{\rho}|<\frac{1}{2}\bar{\rho}$, then
\begin{equation}\label{rho01/2}
c_2\bar{\rho}^{\gamma-2}|\rho-\bar{\rho}|^{2}
\leq
\rho^{\gamma}-\bar{\rho}^{\gamma}
-\gamma\bar{\rho}^{\gamma-1}(\rho-\bar{\rho}).
\end{equation}
\end{lemma}

\begin{proof}
The case $\rho=\bar\rho=0$ is immediate. For $\rho,\bar\rho>0$, Taylor's integral formula and $\gamma-2<0$ give
\begin{align*}
&\rho^\gamma-\bar\rho^\gamma-\gamma\bar\rho^{\gamma-1}(\rho-\bar\rho)\\
=&\gamma(\gamma-1)(\rho-\bar\rho)^2
  \int_0^1(1-s)\bigl((1-s)\bar\rho+s\rho\bigr)^{\gamma-2}\,ds\\
\ge &\frac{\gamma(\gamma-1)}{2}\bigl(\max\{\rho,\bar\rho\}\bigr)^{\gamma-2}(\rho-\bar\rho)^2.
\end{align*}
By continuity, this bound also holds when exactly one of $\rho,\bar\rho$
is zero.

Under (i), $\rho-\bar\rho\ne0$ and
\[
\max\{\rho,\bar\rho\}\le\bar\rho+|\rho-\bar\rho|\le3|\rho-\bar\rho|,
\]
so the preceding bound yields \eqref{rho0}
with $c_1=\frac{\gamma(\gamma-1)}{2}3^{\gamma-2}$.
Under (ii), $\max\{\rho,\bar\rho\}\le\frac32\bar\rho$, and the same bound
gives \eqref{rho01/2}
with $c_2=\frac{\gamma(\gamma-1)}{2}(3/2)^{\gamma-2}$.
\end{proof}

From \eqref{1.1} and \eqref{3.1}, define
$$
\left\{\begin{array}{l}
w=\rho-\bar{\rho}, \\
z=m-\bar{m},
\end{array}\right.
$$
we have
\begin{eqnarray*}
\left\{\begin{array}{l}
w_t+z_x=0, \\
z_t+\left(\displaystyle\frac{m^2}{\rho}\right)_x+\left(p(\rho)-p(\bar{\rho})\right)_x+\displaystyle\frac{\alpha}{(1+t)^{\nu}} z=-\bar{m}_t ,
\end{array}\right.
\end{eqnarray*}
Define $y=-\int_{-\infty}^x w(r, t) \mathrm{d} r$. The continuity equation then give
$$
y_x=-w, \quad z=y_t .
$$

We next estimate $\bar\rho^{\beta_1}|\bar u|^{\beta_2}$ and $\bar\rho^\delta\bar R/\bar\rho$ on the support of $\bar\rho$, where $\bar{R}=\bar{m}_t+\left(\frac{\bar{m}^2}{\bar{\rho}}\right)_x$. Whenever their norms are used below, these profile functions are defined on $\Omega_t=\{x\in\mathbf{R}:\bar\rho(x,t)>0\}$ and extended by zero outside $\Omega_t$.
\begin{lemma}\label{lem4}
Let $1\leq p<+\infty$, $\Omega_t=\{x\in\mathbf{R}:\bar\rho(x,t)>0\}$,
\begin{equation}\label{beta1}
\beta_1>-\frac{\gamma-1}{p},
\end{equation}
$\beta_2, \beta_3\ge0$, and $\delta\ge0$. Then, it holds that
\begin{equation}\label{rhobeta}
\left\|\bar{\rho}^{\beta_1}|\bar{u}|^{\beta_2}\right\|_{L^p(\Omega_t)}\leq C(1+t)^{-\frac{p[(1+\nu)\beta_1+(\gamma-\nu)\beta_2]-(1+\nu)}{p(\gamma+1)}},
\end{equation}
\begin{equation}\label{rhobeta'}
\left\|\bar{\rho}^{\beta_3}|\bar{u}|^{\beta_2}\right\|_{L^{\infty}(\Omega_t)}\leq C(1+t)^{-\frac{(1+\nu)\beta_3+(\gamma-\nu)\beta_2}{\gamma+1}},
\end{equation}
and
\begin{equation}\label{rhodelta}
\left\|\bar{\rho}^\delta\left(\frac{\bar{R}}{\bar{\rho}}\right)\right\|_{L^p(\Omega_t)} \leq C(1+t)^{-\frac{p[\delta(1+\nu)+2\gamma+1-\nu ]-(1+\nu)}{p(\gamma+1)}}.
\end{equation}
\end{lemma}

\begin{proof}
On the one hand, using \eqref{barrho}, \eqref{ux} and $\xi=x(t+1)^{-\frac{1+\nu}{\gamma+1}}$, we have
\begin{align*}
&\left\|\bar{\rho}^{\beta_1}|\bar{u}|^{\beta_2}\right\|^p_{L^p(\Omega_t)}\\
=&\int_{\Omega_t}\left\{(1+t)^{-\frac{1+\nu}{\gamma+1}}\left[\left(A_0-B_0 (1+t)^{-\frac{2(1+\nu)}{\gamma+1}}x^2\right)_{+}\right]^{\frac{1}{\gamma-1}}\right\}^{\beta_1 p}\left|\frac{x(1+\nu)}{(1+t)(\gamma+1)}\right|^{p\beta_2}dx\\
\leq& C\int^{\sqrt{\frac{A_0}{B_0}}}_{-\sqrt{\frac{A_0}{B_0}}}(1+t)^{-\frac{1+\nu}{\gamma+1}\beta_1 p}(1+t)^{-\frac{\gamma-\nu}{\gamma+1}p\beta_2}
\left[\left(A_0-B_0\xi^2\right)_{+}\right]^{\frac{\beta_1 p}{\gamma-1}}|\xi|^{p\beta_2}(1+t)^{\frac{1+\nu}{\gamma+1}}d\xi\\
\leq&C(1+t)^{-\frac{p[(1+\nu)\beta_1+(\gamma-\nu)\beta_2]-(1+\nu)}{\gamma+1}},
\end{align*}
where the last inequality is from
\begin{align*}
\int^{\sqrt{\frac{A_0}{B_0}}}_{-\sqrt{\frac{A_0}{B_0}}}\left[\left(A_0-B_0\xi^2\right)_{+}\right]^{\frac{\beta_1 p}{\gamma-1}}|\xi|^{p\beta_2}d
\xi&\leq  C\int^{1}_{-1}\left[\left(A_0-A_0\tau^2\right)_{+}\right]^{\frac{\beta_1 p}{\gamma-1}}|\tau|^{p\beta_2} d\tau\\
&\leq C\int^{1}_{0}(1-\tau)^{\frac{\beta_1 p}{\gamma-1}}(1+\tau)^{\frac{\beta_1 p}{\gamma-1}}d\tau\\
&\leq C\int^{1}_{0}(1-\tau)^{\frac{\beta_1 p}{\gamma-1}}d\tau\leq C,
\end{align*}
by $p\beta_2\ge0$, $\sqrt{\frac{A_0}{B_0}}\tau=\xi$, and \eqref{beta1}. Hence, we can obtain \eqref{rhobeta}. In particular,
\begin{align*}
\left\|\bar{\rho}^{\beta_3}|\bar{u}|^{\beta_2}\right\|_{L^{\infty}(\Omega_t)}=\operatorname*{ess\,sup}_{x\in\Omega_t}\left (\bar{\rho}^{\beta_3}|\bar{u}|^{\beta_2}\right)
\leq C(1+t)^{-\frac{(\gamma-\nu)\beta_2+(1+\nu)\beta_3}{\gamma+1}},
\end{align*}
by $|x|\leq \sqrt{\frac{A_0}{B_0}}(1+t)^{\frac{1+\nu}{\gamma+1}}$, which gives \eqref{rhobeta'}.

On the other hand, since $\bar{R}=\bar{m}_t+\left(\frac{\bar{m}^2}{\bar{\rho}}\right)_x$, and $\bar{m}=\bar{\rho}\bar{u}$, we have
\begin{align*}
\bar{R}&=(\bar{\rho}\bar{u})_t+\left(\bar{\rho}\bar{u}^2\right)_x=\bar{\rho}_t\bar{u}+\bar{\rho}\bar{u}_t+\bar{\rho}_x\bar{u}^2+
2\bar{\rho}\bar{u}\bar{u}_x=\bar{\rho}\bar{u}_t+\bar{\rho}\bar{u}\bar{u}_x,
\end{align*}
by $\bar{\rho}_t+(\bar{\rho}\bar{u})_x=0$. Combining the above equality and \eqref{ux}, we derive
$$
\frac{\bar{R}}{\bar{\rho}}=\bar{u}_t+\bar{u}\bar{u}_x=-\frac{(1+\nu)(\gamma-\nu) }{(1+t)^{2}(\gamma+1)^2}x,
$$
which implies that with $\xi=x(t+1)^{-\frac{1+\nu}{\gamma+1}}$,
\begin{align*}
&\left\|\bar{\rho}^\delta\left(\frac{\bar{R}}{\bar{\rho}}\right)\right\|^{p}_{L^p(\Omega_t)} \\ =&\int_{\Omega_t}\left\{(1+t)^{-\frac{1+\nu}{\gamma+1}}\left[\left(A_0-B_0 (1+t)^{-\frac{2(1+\nu)}{\gamma+1}}x^2\right)_{+}\right]^{\frac{1}{\gamma-1}}\right\}^{\delta p}\left|\frac{(1+\nu)(\gamma-\nu)x}{(1+t)^{2}(\gamma+1)^2}\right|^{p}dx\\
\leq& C\int^{\sqrt{\frac{A_0}{B_0}}}_{-\sqrt{\frac{A_0}{B_0}}}(1+t)^{-\frac{(1+\nu)\delta p}{\gamma+1}}(1+t)^{-\frac{p(2\gamma+1-\nu)}{\gamma+1}}
\left[\left(A_0-B_0\xi^2\right)_{+}\right]^{\frac{\delta p}{\gamma-1}}|\xi|^p(1+t)^{\frac{1+\nu}{\gamma+1}}d\xi\\
\leq& C(1+t)^{-\frac{p[\delta(1+\nu)+2\gamma+1-\nu]-(1+\nu)}{\gamma+1}},
\end{align*}
which proves \eqref{rhodelta}.
\end{proof}

\section{Extension of key intermediate estimates}

In this section, we extend the results in \cite[Lemmas 3.8 and 3.9]{GEN3} (with $\gamma\in(1,3)$) to $\gamma\in(1,+\infty)$. In \cite{GEN3}, the authors used the theory of divergence-measure fields (see Chen and Frid \cite{CHE1}),  and the following  entropy inequality
$$
\tilde{\eta}_{t}+\tilde{q}_{x}+\frac{2\alpha C_2m^2}{(1+t)^{\nu}}+\frac{\alpha A_mm}{(1+t)^{\nu}}\leq0,
$$
where
\begin{align*}
\tilde{\eta}=\rho \int_{-1}^1 \left|u+z \rho^\theta\right|^{\frac{2\gamma}{\gamma-1}}\left(1-z^2\right)^\lambda d z=C_1\rho^{\gamma+1}+C_2m^2+A(\rho,m),
\end{align*}
with
\begin{equation}\label{C1C2}
\begin{aligned}
& C_1=\int_{-1}^1|z|^{\frac{2 \gamma}{\gamma-1}}\left(1-z^2\right)^\lambda d z, \\
& C_2=\frac{\gamma(\gamma+1)}{(\gamma-1)^2} \int_{-1}^1|z|^{\frac{2}{\gamma-1}}\left(1-z^2\right)^\lambda d z=\frac{2 \gamma(\gamma+1)}{(\gamma-1)^2} C_1,
\end{aligned}
\end{equation}
and
\begin{equation}\label{Amm3A}
A_mm\ge 3A(\rho,m)\ge0.
\end{equation}

We find that the restriction $1<\gamma<3$ comes from \eqref{Amm3A}. Hence, we need to re-estimate $\tilde{\eta}$ and \eqref{Amm3A}, and begin with Taylor's integral formula.

\begin{lemma}\label{lem41}
Let $f(\xi)=|\xi|^{k}$ with $ k>2$. Then,
$$
f(u+z)-f(z)-f'(z)u=u^{2}\int^{1}_{0}(1-s)f^{''}(su+z)ds.
$$
\end{lemma}
The next lemma establishes the required entropy decomposition.
\begin{lemma}\label{lem42}
For $1<\gamma<+\infty$ and $g(\xi)=|\xi|^{\frac{2\gamma}{\gamma-1}}$, it holds that
$$
\eta=\rho \int_{-1}^1 g\left(u+z \rho^\theta\right)\left(1-z^2\right)^\lambda dz=C_1\rho^{\gamma+1}+C_2m^2+D(\rho,m),
$$
where $C_1$, $C_2$ are defined in \eqref{C1C2}, and
$$
m\partial_m D\ge 2D(\rho,m)\ge0.
$$
\end{lemma}

\begin{proof}

According to \eqref{etaq} and Lemma \ref{lem41}, for $g(\xi)=|\xi|^{\frac{2\gamma}{\gamma-1}}$, we obtain
\begin{equation}\label{etac2m}
\begin{aligned}
\eta &=\rho \int_{-1}^1 g\left(u+z \rho^\theta\right)\left(1-z^2\right)^\lambda dz\\
&=\rho \int_{-1}^1\left[g(z \rho^\theta)+g'(z \rho^\theta)u+u^2\int^{1}_{0}(1-s)g''(su+z \rho^\theta)ds\right]\left(1-z^2\right)^\lambda dz\\
&=C_1\rho^{\gamma+1}+\frac{2\gamma(\gamma+1)}{(\gamma-1)^2}\rho u^2\int_{-1}^1\int^{1}_{0}(1-s)|su+z \rho^\theta|^{\frac{2}{\gamma-1}}\left(1-z^2\right)^\lambda ds dz\\
&=C_1\rho^{\gamma+1}+\frac{2\gamma(\gamma+1)}{(\gamma-1)^2}\rho^2 u^2\int_{-1}^1\int^{1}_{0}(1-s)\left|s\frac{u}{\rho^\theta}+z \right|^{\frac{2}{\gamma-1}}\left(1-z^2\right)^\lambda ds dz,
\end{aligned}
\end{equation}
due to the fact that
$$
\int_{-1}^1g'(z \rho^\theta)\left(1-z^2\right)^\lambda dz=0.
$$
Then, we set
\begin{align*}
h(b)&=\int_{-1}^1\left|b+z \right|^{\frac{2}{\gamma-1}}\left(1-z^2\right)^\lambda dz\\
&=\int_{0}^1\left|b+z \right|^{\frac{2}{\gamma-1}}\left(1-z^2\right)^\lambda dz+\int_{-1}^0\left|b+z \right|^{\frac{2}{\gamma-1}}\left(1-z^2\right)^\lambda dz\\
&=\int_{0}^1\left|b+z \right|^{\frac{2}{\gamma-1}}\left(1-z^2\right)^\lambda dz+\int_{0}^1\left|b-z \right|^{\frac{2}{\gamma-1}}\left(1-z^2\right)^\lambda dz,
\end{align*}
which implies that
$$
h(b)=h(-b).
$$
That is, $h$ is an even function of $b$. It is obvious that
$$h(0)=\int_{-1}^1|z |^{\frac{2}{\gamma-1}}\left(1-z^2\right)^\lambda dz=\frac{(\gamma-1)^2}{\gamma(\gamma+1)}C_2.$$
Hence, with $a=\frac{u}{\rho^\theta}$ and \eqref{etac2m}, we have
\begin{align*}
\eta&=C_1\rho^{\gamma+1}+\frac{2\gamma(\gamma+1)}{(\gamma-1)^2}\rho^2 u^2\int_{-1}^1\int^{1}_{0}(1-s)\left|sa+z \right|^{\frac{2}{\gamma-1}}\left(1-z^2\right)^\lambda ds dz\\
&=C_1\rho^{\gamma+1}+\frac{2\gamma(\gamma+1)}{(\gamma-1)^2}m^2\int^{1}_{0}(1-s)h(sa)ds\\
&=C_1\rho^{\gamma+1}+C_2m^2+m^2\left[\frac{2\gamma(\gamma+1)}{(\gamma-1)^2}\int^{1}_{0}(1-s)[h(sa)-h(0)]ds\right].
\end{align*}
Then, we show
\begin{equation}\label{Drhom}
D(\rho, m)=m^2\left[\frac{2\gamma(\gamma+1)}{(\gamma-1)^2}\int^{1}_{0}(1-s)[h(sa)-h(0)]ds\right]\ge0,
\end{equation}
and set $\sigma=\frac{2}{\gamma-1}$ for simplicity. Then,
$$h(b)=\int_{-1}^1\left|b+z \right|^{\sigma}\left(1-z^2\right)^{\frac{\sigma-1}{2}} dz.$$
Since $h$ is even, we only consider $b\ge0$ and divide the proof into two cases.

\vspace{4pt}
\noindent\textbf{$Case~ I: \sigma\ge1.$}

In this case, $1<\gamma\le3$, and $b\mapsto|b+z|^\sigma$ is convex. Thus $h$ is an even convex function and is nondecreasing on
$[0,\infty)$, which implies that $h(b)\ge h(0)$ if $b\ge0$.

\vspace{4pt}
\noindent\textbf{$Case~ II: \sigma<1.$}

In this case $\gamma>3$. For $0\le b<1$, applying \cite[Lemma~2.4]{GUT} with
$d=R=1$, $\alpha=\ell=\sigma$, $x=b$ and $y=-z$
in the notation of that reference, we obtain
$$
h(b)
=
B\left(\frac{\sigma+1}{2},\frac{\sigma+1}{2}\right)
{}_2F_1\left(-\frac{\sigma}{2},-\sigma;
            \frac12;b^2\right)
=
h(0)\,{}_2F_1\left(-\frac{\sigma}{2},-\sigma;
                  \frac12;b^2\right).
$$
Here $B$ denotes the beta function, and
\[
h(0)=B\left(\frac{\sigma+1}{2},
           \frac{\sigma+1}{2}\right).
\]

For $0\le b<1$, its convergent power series is
\[
\frac{h(b)}{h(0)}
=1+\sum_{n=1}^{\infty}
  \frac{(-\sigma/2)_n(-\sigma)_n}{(1/2)_n\,n!}\,b^{2n},
\]
where $(a)_n=a(a+1)\cdots(a+n-1)$ and $(a)_0=1$.
For $0<\sigma<1$, both $(-\sigma/2)_n$ and $(-\sigma)_n$ are strictly
negative for every $n\ge1$. Thus all displayed coefficients are positive,
so $h(b)\ge h(0)$ and $h$ is strictly increasing on $(0,1)$.
Continuity, which follows from dominated convergence in the defining
integral, extends this conclusion to $b=1$.

For $b>1$, differentiation under the integral sign is valid and gives
\[
h'(b)=\sigma\int_{-1}^1
             (b+z)^{\sigma-1}(1-z^2)^{(\sigma-1)/2}\,dz>0.
\]
Thus $h$ is nondecreasing on $[0,\infty)$ and $h(b)\ge h(0)$ for every $b\in(-\infty,+\infty)$ by evenness. Differentiability at $b=1$ is not required.
Substituting this inequality into \eqref{Drhom} yields $D(\rho,m)\ge0$.

Similarly, for fixed $\rho>0$, the function
\[
 m\longmapsto \frac{D(\rho,m)}{m^2}
 =\frac{2\gamma(\gamma+1)}{(\gamma-1)^2}
 \int_0^1(1-s)\left[h\left(\frac{s m}{\rho^{1+\theta}}\right)-h(0)\right]ds
\]
is even and nondecreasing in $|m|$ (for $m\ne0$). It is
continuously differentiable as a function of $m$, as follows either
from the defining entropy integral and the $C^2$ regularity of
$g(\xi)=|\xi|^{\frac{2\gamma}{\gamma-1}}$, or from the Taylor-remainder representation.
Consequently,
\[
 m\frac{\partial}{\partial m}
   \left(\frac{D(\rho,m)}{m^2}\right)\ge0,
 \qquad m\ne0,
\]
which is exactly
\[
 m\partial_m D(\rho,m)-2D(\rho,m)\ge0.
\]
At $m=0$ we have $D(\rho,0)=0$, and the inequality follows by continuity. This completes the proof.
\end{proof}

\begin{remark}\label{remark4.3}
\rm{The coefficient $2$ in Lemma~\ref{lem42} suffices for the weighted estimates below.}

\rm{In the proofs of \cite[Lemmas~3.3--3.5]{GEN3}, the restriction $1<\gamma<3$ enters through the structural inequality $mA_m\ge3A\ge0$, whereas the remaining estimates require only $\gamma>1$. Lemma~\ref{lem42} provides the substitute $m\partial_m D\ge2D\ge0$ for all $\gamma>1$. For $0<\nu<1$, this inequality is sufficient to recover the basic energy estimates of \cite[Lemmas~3.3--3.5]{GEN3}, with $A$ replaced by $D$, under the corresponding assumptions on the initial data. We shall use these estimates below without repeating their proofs.}

{\rmfamily
We explain how Lemma~\ref{lem42} applies to the weighted estimates of
\cite[Lemmas~3.3, 3.6 and 3.9]{GEN3}. Multiplication by
$w_r(t)=(1+t)^r$ gives the dissipation term
$\alpha(1+t)^{r-\nu}m\,\partial_m D$, while differentiation of the weight
contributes $r(1+t)^{r-1}D$ on the right-hand side. By Lemma~\ref{lem42},
\[
\alpha(1+t)^{r-\nu}m\,\partial_m D-r(1+t)^{r-1}D
\ge (1+t)^{r-\nu}\bigl[2\alpha-r(1+t)^{\nu-1}\bigr]D .
\]
Since $0\le\nu<1$, for each fixed $r$ we can choose $T_r\ge0$ such that
$|r|(1+T_r)^{\nu-1}\le\alpha$. For $t\ge T_r$, the right-hand side
is therefore bounded below by $\alpha(1+t)^{r-\nu}D$.
On $[0,T_r]$, the weights are bounded and $0\le D\le\tilde\eta$,
so the unweighted higher-order entropy estimate controls the finite-time
contribution. This gives the required absorption for $r=\nu$ in the basic
estimate and for $r=\mu_{k+1}(\varepsilon)$ in the iteration.}
\end{remark}

We first record a weighted mechanical energy estimate, adapted from \cite[Lemma~3.7]{GEN3}, which will be used in the iteration argument below.

\begin{lemma}\label{lem44}
For $\nu\in[0,1)$, $\gamma>1$ and any $\beta_0>-1$, it holds that
\begin{equation}\label{omegabeta}
\begin{aligned}
&(1+t)^\omega \int \eta_e d x+\frac{\alpha}{2} \int_0^t \int(1+\tau)^{\omega-\nu} \frac{m^2}{\rho} d x d \tau \\
& \leq C+C \int_0^t \int(1+\tau)^{\beta_0}|\rho-\bar{\rho}|^{\gamma+1} d x d \tau,
\end{aligned}
\end{equation}
where
$$\eta_e=\frac{\rho}{2J_\lambda} \int_{-1}^1 \left|u+z \rho^\theta\right|^2\left(1-z^2\right)^\lambda d z=\frac{1}{2}\frac{m^2}{\rho}+\frac{\kappa}{\gamma-1}\rho^{\gamma},$$
and $\omega<\min \left\{\frac{(\gamma-1)(\nu+1)}{\gamma+1}, \frac{(\gamma-1)(1+\beta_0)}{\gamma}\right\}$.
\end{lemma}

In \eqref{omegabeta}, the mechanical entropy dissipates at rate $\alpha m^2/[\rho(1+t)^\nu]$. Choose $T$ so that $|\omega|(1+T)^{\nu-1}\le\alpha$; the kinetic weight derivative is then absorbed with remainder at least $\alpha m^2/[2\rho(1+t)^\nu]$. Contributions on $[0,T]$ enter $C$. Constants may depend on the fixed $\alpha=\kappa>0$; no uniformity as $\gamma\downarrow1$ is asserted.

The proof of \cite[Lemma~3.7]{GEN3} also applies to $\beta_0>-1$
and $\nu=0$. Indeed, the time-integrability conditions
used there are
\[
\gamma(\omega-1)-(\gamma-1)\beta_0<-1,
\qquad
\omega<\frac{(\gamma-1)(1+\nu)}{\gamma+1},
\]
both of which follow from the stated restriction on
$\omega$. In particular, this allows the choice
$\beta_0=-\varepsilon$ needed below when $\nu=0$.

We next collect the weighted estimates and recursive relations adapted from \cite[Lemmas~3.6, 3.8 and~3.9]{GEN3}. As explained in Remark \ref{remark4.3}, the structural inequality established in Lemma \ref{lem42} allows the corresponding arguments to be extended to all $\gamma>1$.

\begin{lemma}\label{lem45}
For $\nu\in[0,1)$, $\gamma>1$, it holds that
\begin{align*}
& (1+t)^{\mu_1(\varepsilon)} \int|\rho-\bar{\rho}|^{\gamma+1} d x+(1+t)^{\mu_1(\varepsilon)} \int m^2 d x \\
& +\int_0^t \int(1+\tau)^{\mu_1(\varepsilon)-\nu} y_{\tau}^2 d x d \tau  +\int_0^t \int(1+\tau)^{\mu_1(\varepsilon)-\nu} D d x d \tau \leq C,
\end{align*}
and
\begin{equation}\label{theta1}
\begin{aligned}
& (1+t)^{\theta_1(\varepsilon)-\nu} \int y^2 d x+\int_0^t \int(1+\tau)^{\theta_1(\varepsilon)}\left(\rho^\gamma-\bar{\rho}^\gamma\right)(\rho-\bar{\rho}) d x d \tau \\
& +\int_0^t \int(1+\tau)^{\theta_1(\varepsilon)} m^2 d x d \tau \\
& +\left(\nu-\theta_1(\varepsilon)\right) \int_0^t \int(1+\tau)^{\theta_1(\varepsilon)-\nu-1} y^2 d x d \tau \leq C,
\end{aligned}
\end{equation}
where, for sufficiently small $\varepsilon>0$,
$$
\mu_1(\varepsilon)=\min \left\{1,1+\frac{\nu}{2}-\frac{\nu+1}{2(\gamma+1)}\right\}-\varepsilon:=\tilde{\mu}_1-\varepsilon,
$$
and
$$\theta_1(\varepsilon)=\min \left\{\tilde{\mu}_1-\nu, \nu, \frac{\gamma-\nu}{\gamma+1}\right\}-\varepsilon=: \tilde{\theta}_1-\varepsilon.$$

Furthermore, it holds that
\begin{equation}\label{muk}
\begin{gathered}
(1+t)^{\mu_{k+1}(\varepsilon)} \int|\rho-\bar{\rho}|^{\gamma+1} d x+(1+t)^{\mu_{k+1}(\varepsilon)} \int m^2 d x \\
+\int_0^t \int(1+\tau)^{\mu_{k+1}(\varepsilon)-\nu} y_{\tau}^2 d x d \tau \\
+\int_0^t \int(1+\tau)^{\mu_{k+1}(\varepsilon)-\nu} D d x d \tau \leq C,
\end{gathered}
\end{equation}
with
\begin{equation}\label{muk+1}
\mu_{k+1}(\varepsilon)=\tilde{\mu}_{k+1}-\varepsilon=\min \left\{1+\tilde{\theta}_k, 1+\frac{\nu}{2}-\frac{\nu+1}{2(\gamma+1)}+\frac{\tilde{\theta}_k}{2}\right\}-\varepsilon, \quad \forall k \in \mathbf{N},
\end{equation}
and
\begin{equation}\label{thetak}
\tilde{\theta}_k=\min \left\{\tilde{\mu}_k-\nu, \nu, \frac{\gamma-\nu}{\gamma+1}\right\}.
\end{equation}
The sequences $\{\tilde\mu_k\}$ and $\{\tilde\theta_k\}$ are nondecreasing, with $\tilde{\theta}_0=0$. Moreover, for every integer $k \geq 1$, define
$$
\theta_k(\varepsilon):=\tilde{\theta}_k-\varepsilon .
$$
Estimate \eqref{theta1} remains valid with $\theta_1(\varepsilon)$ replaced by $\theta_k(\varepsilon)$, with a constant $C_{k, \varepsilon}$ independent of $t$.
\end{lemma}

To verify the iteration step, fix $k\ge1$ and set
$r=\mu_{k+1}(\varepsilon)$ and $\theta_k=\tilde\theta_k-\varepsilon$.
The two bounds in \eqref{muk+1} give, respectively,
\[
r-1\le\theta_k,\qquad
2r-\nu-2+\frac{1+\nu}{\gamma+1}-\theta_k\le-\varepsilon .
\]
The first inequality allows the pressure and momentum terms arising from
differentiation of the weight to be estimated by \eqref{theta1} with
exponent $\theta_k$. The second bounds the time factor in the profile
terms by $(1+t)^{-\varepsilon}$, as in the proof of
\cite[Lemma~3.9]{GEN3}. Remark~\ref{remark4.3} absorbs the remaining term
\[
r\int_0^t\int (1+\tau)^{r-1}D(\rho,m)\,dx\,d\tau
\]
into the weighted dissipation, up to a finite-time contribution.
This yields \eqref{muk} with a constant independent of $t$.

For the converse step, put $s=\theta_{k+1}(\varepsilon)$ and
$b=(1+\nu)/(\gamma+1)$. By \eqref{thetak},
\[
 s\le r-\nu,\qquad s\le\nu-\varepsilon,\qquad
 s\le(\gamma-\nu)/(\gamma+1)-\varepsilon .
\]
Apply the weighted primitive estimate of \cite[Lemma~3.8]{GEN3}
with multiplier $(1+t)^s y$. The first inequality bounds the integral
of $(1+t)^s y_t^2$ by \eqref{muk}. Differentiating the damping weight
$(1+t)^{s-\nu}$ gives the dissipative coefficient $\nu-s>0$ in
\eqref{theta1}. For the mixed kinetic term, $\bar\rho\le C(1+t)^{-b}$,
so the required mechanical dissipation has weight $s-b$.
Take $\beta_0=s$ and $\omega=s+\nu-b$ in Lemma~\ref{lem44}.
Indeed, $s>-1$ for small $\varepsilon$, and the two upper bounds on
$\omega$ there are equivalent, respectively, to
\[
 s<\frac{\gamma-\nu}{\gamma+1},\qquad
 s<\frac{\gamma-\nu}{\gamma+1}+(\gamma-1)(1-\nu).
\]
Both follow from the third inequality above. In the Young inequality
in the proof of Lemma~\ref{lem44}, choose the coefficient of
$\int_0^t\int(1+\tau)^s|\rho-\bar\rho|^{\gamma+1}\,dx\,d\tau$
sufficiently small, as in \cite[Lemmas~3.7--3.8]{GEN3}.
Since $|\rho-\bar\rho|^{\gamma+1}\le
(\rho^\gamma-\bar\rho^\gamma)(\rho-\bar\rho)$, this term is absorbed
by the pressure integral. This gives \eqref{theta1} with exponent $s$
and closes the induction. The constant may depend on $k$ and
$\varepsilon$, but is uniform in time.

\begin{remark}\label{remark4.6}
Although \cite{GEN3} considers a positive damping exponent, the estimates in Lemmas~\ref{lem44} and~\ref{lem45} also hold for $\nu=0$. In this case, the
basic estimates follow from the constant-damping theory in \cite{GEN1, HUA4}, and the weighted entropy argument in Lemma~3.6 of \cite{GEN3}
applies using Lemma~\ref{lem42}. Moreover, $\widetilde{\theta}_k=0$ and $\mu_k(\varepsilon)=1-\frac{1}{2(\gamma+1)}-\varepsilon$ for every $k\geq1$, so the first weighted estimate already gives the required exponent and no further iteration is needed.
\end{remark}

We now analyze the recursive exponents in Lemma \ref{lem45} to obtain the decay and weighted space-time estimates needed in the proof of Theorem \ref{thm1}.

\begin{theorem}\label{thm3}
Under the conditions of Theorem \ref{thm1}, let $\nu\in[0,1)$ and $\gamma>1$. Then, for every sufficiently small $\varepsilon>0$, it holds that
\begin{equation}\label{lem9gen3}
\begin{aligned}
& (1+t)^{\mu^{*}(\varepsilon)} \int|\rho-\bar{\rho}|^{\gamma+1} d x
+  \int_0^t \int(1+\tau)^{\mu^{*}(\varepsilon)-\nu} y_{\tau}^2 d x d \tau\\
&+\int_0^t \int(1+\tau)^{\theta^{*}}\left(\rho^{\gamma}-\bar{\rho}^{\gamma}\right)(\rho-\bar{\rho})dxd\tau \leq C,
\end{aligned}
\end{equation}
where
\begin{eqnarray}\label{1.55}
\mu^{*}(\varepsilon):=\left\{\begin{array}{ll}
1+\nu-\frac{\nu+1}{2(\gamma+1)}-\varepsilon, & \nu \in\left[0, \frac{\gamma}{\gamma+2}\right], \\
\frac{3}{2}+\frac{\nu}{2}-\frac{\nu+1}{\gamma+1}-\varepsilon, & \nu \in\left[\frac{\gamma}{\gamma+2}, 1\right),
\end{array} \right.
\end{eqnarray}
and
\begin{eqnarray}\label{1.56}
\theta^{*}:=\left\{\begin{array}{ll}
\nu-\varepsilon , & \nu \in\left[0, \frac{\gamma}{\gamma+2}\right], \\
\frac{\gamma-\nu}{\gamma+1}-\varepsilon, & \nu \in\left[\frac{\gamma}{\gamma+2}, 1\right).
\end{array} \right.
\end{eqnarray}

\end{theorem}

\begin{proof}
We focus on the case $ \nu \in\left[\frac{\gamma}{\gamma+2}, 1\right)$. For $0<\nu\le\gamma/(\gamma+2)$, the conclusion follows by applying the corresponding iteration argument in \cite{GEN3}, using the extensions of \cite[Lemmas~3.3--3.9]{GEN3} to all $\gamma>1$ established above. The case $\nu=0$ follows from the preceding Remark \ref{remark4.6}.

It follows from \eqref{muk+1} and \eqref{thetak} that
\begin{equation}\label{tildemuk}
\begin{gathered}
\tilde{\mu}_{k+1}= \begin{cases}1+\tilde{\theta}_k, & \tilde{\theta}_k \in\left(-\infty, \displaystyle\frac{\nu \gamma-1}{\gamma+1}\right], \\[0.36cm]
1+\frac{\nu}{2}-\frac{\nu+1}{2(\gamma+1)}+\frac{\tilde{\theta}_k}{2}, & \tilde{\theta}_k \in\left(\displaystyle\frac{\nu \gamma-1}{\gamma+1}, \infty\right),\end{cases}
\end{gathered}
\end{equation}
and
\begin{equation}\label{tildathetak}
\tilde{\theta}_k=\min \left\{\tilde{\mu}_k-\nu, \frac{\gamma-\nu}{\gamma+1}\right\}, \quad \nu \in\left[\frac{\gamma}{\gamma+2},1\right),
\end{equation}
since $\nu\ge \frac{\gamma-\nu}{\gamma+1}$ provided that $\nu \in\left[ \frac{\gamma}{\gamma+2},1\right)$.

First, we claim that there exists $k_1 \in \mathbf{N}$ such that for any $k \geq k_1$,
\begin{equation}\label{tildathetak'}
\tilde{\theta}_k \geq \tilde{\theta}_{k_1}>\frac{\nu \gamma-1}{\gamma+1} .
\end{equation}
If not, it follows from \eqref{tildathetak} that
\begin{equation}\label{0tildek}
0<\tilde{\theta}_k=\tilde{\mu}_k-\nu \leq \frac{\nu \gamma-1}{\gamma+1}<\frac{\gamma-\nu}{\gamma+1}, \quad \forall k \in \mathbf{N},
\end{equation}
Combining this and \eqref{tildemuk} shows
$$
\tilde{\mu}_{k+1}=1+\tilde{\theta}_k=1+\tilde{\mu}_k-\nu=k(1-\nu)+\tilde{\mu}_1 \rightarrow+\infty, ~ \text { as } ~ k \rightarrow+\infty ,
$$
which contradicts \eqref{0tildek}, and hence \eqref{tildathetak'} holds.

Next, we claim that there exists $k_2 \geq k_1$ such that
\begin{equation}\label{tildemuk2}
\tilde{\mu}_{k_2}-\nu \geq \frac{\gamma-\nu}{\gamma+1} .
\end{equation}
If not, then
\begin{equation}\label{tildethetak2}
\tilde{\theta}_k=\tilde{\mu}_k-\nu<\frac{\gamma-\nu}{\gamma+1}, \quad \forall k \geq k_1.
\end{equation}
Noting the above inequality, \eqref{tildemuk} and \eqref{tildathetak'}, we deduce that
\begin{align*}
\tilde{\mu}_{k+1}-\nu&=1+\frac{\nu}{2}-\frac{\nu+1}{2(\gamma+1)}+\frac{\tilde{\mu}_k-\nu}{2}-\nu\\
&=\left(1+\frac{1}{2}\right)\left(1+\frac{\nu}{2}-\frac{\nu+1}{2(\gamma+1)}-\nu\right)+\frac{1}{2^2}(\tilde{\mu}_{k-1}-\nu)\\
&=\left(1+\frac{1}{2}+\cdot\cdot\cdot+\frac{1}{2^{k-k_1}}\right)\left(1+\frac{\nu}{2}-\frac{\nu+1}{2(\gamma+1)}-\nu\right)+
\frac{1}{2^{k+1-k_1}}(\tilde{\mu}_{k_1}-\nu)\\
&\ge \left(1+\frac{1}{2}+\cdot\cdot\cdot+\frac{1}{2^{k-k_1}}\right)\left(1-\frac{\nu}{2}-\frac{\nu+1}{2(\gamma+1)}\right).
\end{align*}
Since
\begin{align*}
2\left(1-\frac{\nu}{2}-\frac{\nu+1}{2(\gamma+1)}\right)-\frac{\gamma-\nu}{\gamma+1}=1-\nu>0,
\end{align*}
we know that for $k_3$ sufficiently large, it holds that $$\tilde{\mu}_{k_3+1}-\nu>\frac{\gamma-\nu}{\gamma+1},$$
which contradicts \eqref{tildethetak2} and consequently \eqref{tildemuk2} holds.

Hence, from \eqref{tildathetak} and \eqref{tildemuk2}, it follows that
$$
\theta^*=\lim_{k\rightarrow\infty}\tilde{\theta}_k-\varepsilon=\frac{\gamma-\nu}{\gamma+1}-\varepsilon,
$$
which, together with \eqref{tildemuk} and \eqref{tildathetak'}, justifies the conclusion.
\end{proof}

\begin{remark}{\rm
The last term on the left-hand side of \eqref{lem9gen3}, obtained from \eqref{theta1} and \eqref{muk}, provides the weighted space-time bound needed to estimate $I_1$, $I_3$, $\hat I_1$ and $\hat I_3$ in the next section. In particular,
\begin{align*}
\int_0^t \int(1+\tau)^{\theta^{*}}|\rho-\bar{\rho}|^{\gamma+1} d xd\tau\leq C\int_0^t \int(1+\tau)^{\theta^{*}}\left|\rho^{\gamma}-\bar{\rho}^{\gamma}\right||\rho-\bar{\rho}|dxd\tau \leq C
\end{align*}
provides time-integrability that does not follow from the pointwise estimate
$$
(1+t)^{\mu^{*}(\varepsilon)} \int|\rho-\bar{\rho}|^{\gamma+1} d x\leq C,
$$
since
$$
\mu^{*}(\varepsilon)< \theta^{*}+1.
$$}
\end{remark}

Observing \eqref{omegabeta} and \eqref{lem9gen3}, and choosing $\beta_0=\theta^*$, we obtain
\begin{equation}\label{omeganu}
(1+t)^\omega \int \eta_e d x+\frac{\alpha}{2} \int_0^t \int(1+\tau)^{\omega-\nu} \frac{m^2}{\rho} d x d \tau\leq C,
\end{equation}
with $0<\omega<\min \left\{\frac{(\gamma-1)(\nu+1)}{\gamma+1}, \frac{(\gamma-1)(1+\theta^*)}{\gamma}\right\}$. Next, we will show
\begin{equation}\label{omegwtheta}
\frac{(\gamma-1)(\nu+1)}{\gamma+1}<\frac{(\gamma-1)(1+\theta^*)}{\gamma},
\end{equation}
for any $\nu\in[0,1)$. In the case $\nu \in\left[0, \frac{\gamma}{\gamma+2}\right]$, the above inequality is obvious, since $\theta^*=\nu-\varepsilon$. For $\nu \in\left[\frac{\gamma}{\gamma+2}, 1\right)$,
\begin{align*}
\frac{1+\theta^*}{\gamma}-\frac{1+\nu}{\gamma+1}=\frac{1-\nu-\varepsilon}{\gamma}>0,
\end{align*}
which implies that \eqref{omegwtheta} still holds for $\nu \in\left[\frac{\gamma}{\gamma+2}, 1\right)$. Hence, for sufficiently small $\varepsilon>0$, we can obtain \eqref{omeganu} with
\begin{equation}\label{omegavar}
\omega=\frac{(\gamma-1)(\nu+1)}{\gamma+1}-\varepsilon>0,
\end{equation}
which is helpful in the estimate of $I_4$($\hat{I}_4$) defined below.

\section{Proof of the main result}

In this section, we detail the proof of Theorem \ref{thm1}. First, to obtain the $L^{\gamma}$ estimates, we set $v=(\rho, m)^t, \bar{v}=(\bar{\rho}, \bar{m})^t$, where $(\cdot, \cdot)^t$ means the transpose of the vector $(\cdot, \cdot)$, and define
\begin{equation}\label{eta_*}
\begin{aligned}
& \eta_*=\eta(v)-\eta(\bar{v})-\nabla \eta(\bar{v})(v-\bar{v})=\frac{1}{2}Q_*+\frac{\kappa}{\gamma-1}P_*, \\
& q_*=q(v)-q(\bar{v})-\nabla q(\bar{v})(v-\bar{v}),
\end{aligned}
\end{equation}
where $\eta=\eta_e$ and $q$ are obtained from \eqref{etaq} with $g(\xi)=|\xi|^2/(2J_\lambda)$. Thus $q=\frac{m^3}{2\rho^2}+\frac{\kappa\gamma}{\gamma-1}m\rho^{\gamma-1}$. Then,
\begin{equation}\label{pQ}
\begin{aligned}
P_*&=\rho^{\gamma}-\bar{\rho}^{\gamma}-\gamma\bar{\rho}^{\gamma-1}(\rho-\bar{\rho})\ge0,\\
Q_*&=\frac{m^2}{\rho}-\frac{\bar{m}^2}{\bar{\rho}}+\frac{\bar{m}^2}{\bar{\rho}^2}(\rho-\bar{\rho})-\frac{2\bar{m}}{\bar{\rho}}(m-\bar{m})\ge0,
\end{aligned}
\end{equation}
since $\rho^{\gamma}$ and $\frac{m^2}{\rho}$ are convex.

By \cite[(3.8)]{CUI2}, the relative entropy satisfies the following inequality.
\begin{lemma}\label{lem51}
Under the conditions of Theorem \ref{thm1}, let $g(\xi)=|\xi|^2/(2J_\lambda)$. Then, it holds for any $t>0$ and $1<\gamma<+\infty$ that
\begin{equation}\label{detadt'}
\begin{aligned}
\eta_{*t} +q_{*x}+\frac{\alpha Q_{*}}{(1+t)^{\nu}}
\leq\frac{\bar{m}}{\bar{\rho}}\frac{\bar{R}}{\bar{\rho}}\left(\rho-\bar{\rho}\right)-\frac{\bar{R}}{\bar{\rho}}\left(m-\bar{m}\right)+
\frac{\bar{m}}{\bar{\rho}}(\kappa P_*+Q_*)_x.
\end{aligned}
\end{equation}
\end{lemma}

Next, we give the estimate for the relative entropy $\eta_*$.

\begin{theorem}\label{thm4}
Under the assumptions of Theorem \ref{thm1}, for every $\gamma>1$, $t>0$, and sufficiently small $\varepsilon>0$, we have
\begin{equation}\label{phivar}
(1+t)^{\phi(\varepsilon)} \int \eta_* d x +\alpha\int^{t}_{0}\int(1+\tau)^{\phi(\varepsilon)-\nu}Q_* dxd\tau\leq C,
\end{equation}
where $\eta_*$ is defined in \eqref{eta_*}, $Q_*$ is defined in \eqref{pQ} and
\begin{eqnarray}\label{1phi}
\phi(\varepsilon)=\left\{\begin{array}{ll}
\displaystyle\frac{\gamma^2+\gamma-1}{(1+\gamma)^2}(1+\nu)-\varepsilon, & \nu \in\left[0, \displaystyle\frac{\gamma}{\gamma+2}\right], \\[0.36cm]
\displaystyle\frac{2\gamma-1-\nu}{\gamma+1}-\varepsilon, & \nu \in\left[\displaystyle\frac{\gamma}{\gamma+2}, 1\right).
\end{array} \right.
\end{eqnarray}

\end{theorem}

\begin{proof}
We use \eqref{detadt'} in the global weak relative-entropy formulation of \cite[(3.8)]{CUI2},
with the spatial integration justified as in the opening of the proof of
\cite[Lemma~4.5]{CUI2}, using divergence-measure fields \cite{CHE1}. The reference quotients
at vacuum and the resulting integrated products are interpreted in that same
framework.

The time weights in Theorem~\ref{thm3} depend on the damping regime, so we divide the proof into two cases.

\vspace{4pt}
\noindent\textbf{$Case~ I: \nu\in \left[0,\frac{\gamma}{\gamma+2}\right).$}

Setting
\begin{equation}\label{muvar}
\mu(\varepsilon)=\frac{\gamma^2+\gamma-1}{(\gamma+1)^2}(1+\nu)-\varepsilon,
\end{equation}
multiplying \eqref{detadt'} by $(1+t)^{\mu(\varepsilon)}$ and integrating over $(0,t)\times \mathbf{R}$, we have
\begin{equation}\label{etamuv}
\begin{aligned}
&(1+t)^{\mu(\varepsilon)} \int \eta_* d x+\alpha\int^{t}_{0}\int\frac{Q_*}{(1+\tau)^{\nu-\mu(\varepsilon)}} dxd\tau\\
\leq&\int^{t}_{0}\int(1+\tau)^{\mu(\varepsilon)}\left|\frac{\bar{m}}{\bar{\rho}}\frac{\bar{R}}{\bar{\rho}}
\left(\rho-\bar{\rho}\right)\right|dxd\tau+\int^{t}_{0}\int(1+\tau)^{\mu(\varepsilon)} \left|\frac{\bar{R}}{\bar{\rho}}\left(m-\bar{m}\right)\right|dxd\tau\\
&+C\int^{t}_{0}\int\mu(\varepsilon)(1+\tau)^{\mu(\varepsilon)-1}P_*dxd\tau+
C\int^{t}_{0}\int\mu(\varepsilon)(1+\tau)^{\mu(\varepsilon)-1}Q_*d xd\tau+C\\
=&:I_1+I_2+I_3+I_4+C.
\end{aligned}
\end{equation}
Here, we use the spatially integrated relative entropy inequality obtained in the opening step of the proof of Lemma 4.5 of \cite{CUI2}, with the pressure and damping coefficients as in (1.1). In our notation, the term
$$
\frac{1+\nu}{(1+\tau)(\gamma+1)}\int^{t}_{0}\int\left(\kappa P_*+Q_*\right)\chi_{\Omega_\tau} dxd\tau,
$$
is nonnegative and may be dropped from the left-hand side.

To estimate $I_1$, we see from \eqref{rhobeta} and \eqref{rhodelta} that
\begin{align*}
\int\left|\frac{\bar{m}}{\bar{\rho}}\frac{\bar{R}}{\bar{\rho}}
\left(\rho-\bar{\rho}\right)\right|dx
&\leq \bigg|\bigg|\frac{\bar{m}}{\bar{\rho}}\bigg|\bigg|_{L^{\frac{2(\gamma+1)}{\gamma}}}
\bigg|\bigg|\frac{\bar{R}}{\bar{\rho}}\bigg|\bigg|_{L^{\frac{2(\gamma+1)}{\gamma}}}
\big|\big|\rho-\bar{\rho}\big|\big|_{L^{\gamma+1}}\\
&\leq C(1+t)^{-\frac{3\gamma^2+3\gamma+1-\nu(3\gamma+2)}{(\gamma+1)^2}}\big|\big|\rho-\bar{\rho}\big|\big|_{L^{\gamma+1}},
\end{align*}
which gives
\begin{equation}\label{I1}
\begin{aligned}
I_1\leq& \int^{t}_{0}(1+\tau)^{\mu(\varepsilon)-\frac{3\gamma^2+3\gamma+1-\nu(3\gamma+2)}{(\gamma+1)^2}}
\big|\big|\rho-\bar{\rho}\big|\big|_{L^{\gamma+1}}d\tau\\
\leq& \int^{t}_{0}\int(1+\tau)^{\theta^*}\big|\rho-\bar{\rho}\big|^{\gamma+1}dxd\tau\\
&+C\int^{t}_{0}(1+\tau)^{\left[\mu(\varepsilon)-\frac{3\gamma^2+3\gamma+1-\nu(3\gamma+2)}{(\gamma+1)^2}-\frac{\theta^*}{\gamma+1}\right]
\frac{\gamma+1}{\gamma}}d\tau\\
\leq& C,
\end{aligned}
\end{equation}
due to \eqref{lem9gen3} and
\begin{align*}
&\mu(\varepsilon)-\frac{3\gamma^2+3\gamma+1-\nu(3\gamma+2)}{(\gamma+1)^2}-\frac{\theta^*}{\gamma+1}+\frac{\gamma}{\gamma+1}\\
=&\frac{\gamma^2+\gamma-1}{(\gamma+1)^2}(1+\nu)-\frac{\gamma}{\gamma+1}\varepsilon-\frac{2\gamma^2+2\gamma+1-\nu(2\gamma+1)}{(\gamma+1)^2}\\
=&\frac{-(\gamma^2+\gamma+2)+\nu(\gamma^2+3\gamma)}{(\gamma+1)^2}-\frac{\gamma}{\gamma+1}\varepsilon<0,
\end{align*}
by \eqref{1.56} and $\nu\in\left[0,\frac{\gamma}{\gamma+2}\right]$.

As for $I_2$, since $m-\bar{m}=y_t$, we deduce by \eqref{rhodelta} that
\begin{align*}
\int\left|\frac{\bar{R}}{\bar{\rho}}\left(m-\bar{m}\right)\right|dx\leq \bigg|\bigg|\frac{\bar{R}}{\bar{\rho}} \bigg|\bigg|||m-\bar{m}||\leq C(1+t)^{-\frac{4\gamma+1-3\nu}{2(\gamma+1)}}||y_t||,
\end{align*}
which implies
\begin{equation}\label{I2}
\begin{aligned}
I_2\leq& \int^{t}_{0}(1+\tau)^{\mu(\varepsilon)-\frac{4\gamma+1-3\nu}{2(\gamma+1)}}||y_{\tau}||d\tau\\
\leq& \int^{t}_{0}\int(1+\tau)^{\mu^*(\varepsilon)-\nu}y^{2}_{\tau}dxd\tau\\
&+C\int^{t}_{0}(1+\tau)^{2\left[\mu(\varepsilon)-\frac{4\gamma+1-3\nu}{2(\gamma+1)}-\frac{\mu^*(\varepsilon)-\nu}{2}\right]}
d\tau\\
\leq& C,
\end{aligned}
\end{equation}
due to \eqref{lem9gen3} and
\begin{align*}
&\mu(\varepsilon)-\frac{4\gamma+1-3\nu}{2(\gamma+1)}-\frac{\mu^*(\varepsilon)-\nu}{2}+\frac{1}{2}\\
&=\frac{\gamma^2+\gamma-1}{(\gamma+1)^2}(1+\nu)-\frac{\varepsilon}{2}-\frac{8\gamma+1-7\nu}{4(\gamma+1)}\\
&=-\frac{4\gamma^2+5\gamma+5-\nu(4\gamma^2+11\gamma+3)}{4(\gamma+1)^2}-\frac{\varepsilon}{2}<0,
\end{align*}
by \eqref{1.55} and $\nu\in\left[0,\frac{\gamma}{\gamma+2}\right]$.

Then, to estimate $I_3$, using Lemma \ref{lem31}, we first get
\begin{align*}
P_*=\rho^{\gamma}-\bar{\rho}^{\gamma}-\gamma\bar{\rho}^{\gamma-1}(\rho-\bar{\rho})\leq C\left[\left(\rho^{\gamma}-\bar{\rho}^{\gamma}\right)(\rho-\bar{\rho})\right]^{\frac{\gamma}{\gamma+1}}.
\end{align*}
Hence, on $\Omega_t=\{x\in\mathbf{R}:\bar\rho(x,t)>0\},$
\begin{align*}
\int_{\Omega_t} P_*dx&=\int_{\Omega_t} \rho^{\gamma}-\bar{\rho}^{\gamma}-\gamma\bar{\rho}^{\gamma-1}(\rho-\bar{\rho}) dx\\
&\leq C\int_{\Omega_t}\left[\left(\rho^{\gamma}-\bar{\rho}^{\gamma}\right)(\rho-\bar{\rho})\right]^{\frac{\gamma}{\gamma+1}}dx\\
&\leq C\left(\int_{\Omega_t}\left(\rho^{\gamma}-\bar{\rho}^{\gamma}\right)(\rho-\bar{\rho})dx\right)^{\frac{\gamma}{\gamma+1}}
\left(\int_{\Omega_t}dx\right)^{\frac{1}{\gamma+1}}\\
&\leq C(1+t)^{\frac{1+\nu}{(\gamma+1)^2}}
\left[\int\left(\rho^\gamma-\bar{\rho}^\gamma\right)(\rho-\bar{\rho})dx\right]^{\frac{\gamma}{\gamma+1}}.
\end{align*}
On $\Omega_t^c$,
$$
\begin{aligned}
\int_{\Omega_t^c} \rho \mathrm{~d} x & =\int_{\Omega_t}(\bar{\rho}-\rho) \mathrm{d} x \\
& \leq \int_{\Omega_t}|\rho-\bar{\rho}| \mathrm{d} x \\
& \leq |\Omega_t|^{\frac{\gamma}{\gamma+1}}\left(\int_{\Omega_t}|\rho-\bar{\rho}|^{\gamma+1} \mathrm{~d} x\right)^{\frac{1}{\gamma+1}} \\
& \leq |\Omega_t|^{\frac{\gamma}{\gamma+1}} \left[\int\left(\rho^\gamma-\bar{\rho}^\gamma\right)(\rho-\bar{\rho})dx\right]^{\frac{1}{\gamma+1}} ,
\end{aligned}
$$
where the last inequality is from Lemma \ref{lem32}. Since $\bar\rho=0$ on $\Omega_t^c$, we have $P_*=\rho^\gamma$ and
$$
\begin{aligned}
\int_{\Omega_t^c} P_* \mathrm{~d} x & =\int_{\Omega_t^c} \rho^\gamma \mathrm{d} x \\
& \leq\left(\int_{\Omega_t^c} \rho \mathrm{~d} x\right)^{1 / \gamma}\left(\int_{\Omega_t^c} \rho^{\gamma+1} \mathrm{~d} x\right)^{(\gamma-1) / \gamma} \\
& \leq |\Omega_t|^{\frac{1}{\gamma+1}} \left[\int\left(\rho^\gamma-\bar{\rho}^\gamma\right)(\rho-\bar{\rho})dx\right]^{\frac{1}{\gamma(\gamma+1)}+\frac{\gamma-1}{\gamma}}\\
&\leq C(1+t)^{\frac{1+\nu}{(\gamma+1)^2}}
\left[\int\left(\rho^\gamma-\bar{\rho}^\gamma\right)(\rho-\bar{\rho})\,dx\right]^{\frac{\gamma}{\gamma+1}} .
\end{aligned}
$$
Consequently, we estimate $I_3$ as follows
\begin{equation}\label{I3}
\begin{aligned}
I_3\leq& C\int^{t}_{0}(1+\tau)^{\mu(\varepsilon)-1+\frac{1+\nu}{(\gamma+1)^2}}
\left(\int
\left(\rho^{\gamma}-\bar{\rho}^{\gamma}\right)\left(\rho-\bar{\rho}\right)  dx\right)^{\frac{\gamma}{\gamma+1}}d\tau\\
\leq& \int^{t}_{0}\int(1+\tau)^{\theta^*}
\left(\rho^{\gamma}-\bar{\rho}^{\gamma}\right)\left(\rho-\bar{\rho}\right)dxd\tau\\
&+C\int^{t}_{0}(1+\tau)^{\left[\mu(\varepsilon)-1+\frac{1+\nu}{(\gamma+1)^2}-\frac{\gamma\theta^*}{\gamma+1}\right](\gamma+1)}
d\tau\\
\leq &C,
\end{aligned}
\end{equation}
due to \eqref{lem9gen3}, $\theta^*=\nu-\varepsilon$ and
\begin{align*}
\mu(\varepsilon)-1+\frac{1+\nu}{(\gamma+1)^2}-\frac{\gamma\theta^*}{\gamma+1}+\frac{1}{\gamma+1}=-\frac{\varepsilon}{\gamma+1}<0.
\end{align*}

Now, we are left to estimate $I_4$. Rewrite $I_4$ as follows
\begin{equation}\label{I412}
\begin{aligned}
I_4&\leq C\int^{t}_{0}\int\mu(\varepsilon)(1+\tau)^{\mu(\varepsilon)-1}Q_*d xd\tau\\
&\leq C\int^{t}_{0}\int(1+\tau)^{\mu(\varepsilon)-1}
\left[\frac{m^2}{\rho}+\frac{\bar{m}^2}{\bar{\rho}}
+\frac{\bar{m}^2}{\bar{\rho}^2}|\rho-\bar{\rho}|+\left|\frac{2\bar{m}}{\bar{\rho}}(m-\bar{m})\right|\right]dxd\tau\\
&=:C\left(I_{41}+I_{42}+I_{43}+I_{44}\right).
\end{aligned}
\end{equation}
First, \eqref{omeganu} and \eqref{omegavar} give the estimate
\begin{equation}\label{I41}
\begin{aligned}
I_{41}=\int^{t}_{0}\int(1+\tau)^{\mu(\varepsilon)-1}\frac{m^2}{\rho}dxd\tau\leq C,
\end{aligned}
\end{equation}
where the last inequality is from
\begin{align*}
\mu(\varepsilon)-1-(\omega-\nu)=\frac{\gamma(1+\nu)}{(\gamma+1)^2}+\nu-1<0,
\end{align*}
by $\nu<\frac{\gamma}{\gamma+2}$.
As for $I_{42}$, it follows from \eqref{rhobeta} that
\begin{equation}\label{I42}
\begin{aligned}
I_{42}=\int^{t}_{0}\int(1+\tau)^{\mu(\varepsilon)-1}\frac{\bar{m}^2}{\bar{\rho}}dxd\tau
\leq\int^{t}_{0}(1+\tau)^{\mu(\varepsilon)-1-\frac{2\gamma-2\nu}{\gamma+1}}d\tau\leq C,
\end{aligned}
\end{equation}
due to $\nu<\frac{\gamma}{\gamma+2}$ and
\begin{align*}
\mu(\varepsilon)-\frac{2\gamma-2\nu}{\gamma+1}\leq&\frac{\gamma^2+\gamma-1}{(\gamma+1)^2}(1+\nu)-\frac{2\gamma-2\nu}{\gamma+1}\\
=&-\frac{\gamma^2+\gamma+1-\nu(\gamma^2+3\gamma+1)}{(\gamma+1)^2}<0.
\end{align*}
Then, to estimate $I_{43}$ and $I_{44}$, using \eqref{rhobeta}, we first have
\begin{align*}
\int\frac{\bar{m}^2}{\bar{\rho}^2}|\rho-\bar{\rho}|dx\leq\big|\big|\bar{u}^2\big|\big|_{L^{\frac{\gamma+1}{\gamma}}}
\big|\big|\rho-\bar{\rho}\big|\big|_{L^{\gamma+1}}\leq C(1+t)^{-\frac{2\gamma^2+\gamma-\nu(3\gamma+2)}{(\gamma+1)^2}}\big|\big|\rho-\bar{\rho}\big|\big|_{L^{\gamma+1}},
\end{align*}
and
\begin{align*}
\int\left|\frac{2\bar{m}}{\bar{\rho}}(m-\bar{m})\right|dx\leq C\big|\big|\bar{u}\big|\big|
\big|\big|m-\bar{m}\big|\big|\leq C(1+t)^{-\frac{2\gamma-1-3\nu}{2(\gamma+1)}}\big|\big|y_t\big|\big|,
\end{align*}
which implies that
\begin{equation}\label{I43}
\begin{aligned}
I_{43}=&\int^{t}_{0}\int(1+\tau)^{\mu(\varepsilon)-1}\left|\frac{\bar{m}^2}{\bar{\rho}^2}(\rho-\bar{\rho})\right|dxd\tau\\
\leq&\int^{t}_{0}(1+\tau)^{\mu(\varepsilon)-1-\frac{2\gamma^2+\gamma-\nu(3\gamma+2)}
{(\gamma+1)^2}}\big|\big|\rho-\bar{\rho}\big|\big|_{L^{\gamma+1}}d\tau\\
\leq&\int^{t}_{0}(1+\tau)^{\left[\mu(\varepsilon)-\frac{3\gamma^2+3\gamma+1-\nu(3\gamma+2)}{(\gamma+1)^2}-\frac{\theta^*}
{\gamma+1}\right]\frac{\gamma+1}{\gamma}}d\tau\\
&+C\int^{t}_{0}\int(1+\tau)^{\theta^*}\big|\rho-\bar{\rho}\big|^{\gamma+1}dxd\tau\\
\leq &C,
\end{aligned}
\end{equation}
and
\begin{equation}\label{I44}
\begin{aligned}
I_{44}=&\int^{t}_{0}\int(1+\tau)^{\mu(\varepsilon)-1}\left|\frac{2\bar{m}}{\bar{\rho}}(m-\bar{m})\right|dxd\tau\\
\leq& C\int^{t}_{0}(1+\tau)^{\mu(\varepsilon)-1-\frac{2\gamma-1-3\nu}{2(\gamma+1)}}\big|\big|y_\tau\big|\big|d\tau\\
\leq&\int^{t}_{0}(1+\tau)^{2\left[\mu(\varepsilon)-\frac{4\gamma+1-3\nu}{2(\gamma+1)}-\frac{\mu^*(\varepsilon)-\nu}
{2}\right]}d\tau\\
&+C\int^{t}_{0}\int(1+\tau)^{\mu^*(\varepsilon)-\nu}y^{2}_{\tau}dxd\tau\\
\leq& C,
\end{aligned}
\end{equation}
where we have used \eqref{lem9gen3}, \eqref{I1} and \eqref{I2}.

Thus, in view of \eqref{I412}-\eqref{I44}, we deduce that
$$
I_{4}\leq C.
$$
Plugging this and \eqref{I1}-\eqref{I3} into \eqref{etamuv}, we immediately obtain \eqref{phivar} for $\nu\in\left[0,\frac{\gamma}{\gamma+2}\right)$.

\vspace{4pt}
\noindent\textbf{$Case~ II: \nu \in \left[\frac{\gamma}{\gamma +2} , 1\right).$}

In this case, set
\begin{equation}\label{hatmuvar}
  \hat{\mu}(\varepsilon)=\frac{2\gamma-1-\nu}{\gamma+1}-\varepsilon.
\end{equation}
Applying the same integrated inequality with the weight $(1+t)^{\hat{\mu}(\varepsilon)}$ and integrating over $(0,t)\times \mathbf{R}$, we have
\begin{equation}\label{etamuv'}
\begin{aligned}
&(1+t)^{\hat{\mu}(\varepsilon)} \int \eta_* d x +\alpha\int^{t}_{0}\int(1+\tau)^{\hat{\mu}(\varepsilon)-\nu}Q_* dxd\tau\\
\leq&\int^{t}_{0}\int(1+\tau)^{\hat{\mu}(\varepsilon)}\left|\frac{\bar{m}}{\bar{\rho}}\frac{\bar{R}}{\bar{\rho}}
\left(\rho-\bar{\rho}\right)\right|dxd\tau+\int^{t}_{0}\int(1+\tau)^{\hat{\mu}(\varepsilon)} \left|\frac{\bar{R}}{\bar{\rho}}\left(m-\bar{m}\right)\right|dxd\tau\\
&+C\int^{t}_{0}\int\hat{\mu}(\varepsilon)(1+\tau)^{\hat{\mu}(\varepsilon)-1}P_*dxd\tau+
C\int^{t}_{0}\int\hat{\mu}(\varepsilon)(1+\tau)^{\hat{\mu}(\varepsilon)-1}Q_*d xd\tau+C\\
=&:\hat{I}_1+\hat{I}_2+\hat{I}_3+\hat{I}_4+C.
\end{aligned}
\end{equation}

First, we estimate $\hat{I}_1$, $\hat{I}_2$  and $\hat{I}_3$. Since the estimates of $\frac{\bar{R}}{\bar{\rho}}$, $\bar{\rho}$ and $\frac{\bar{m}}{\bar{\rho}}$ are the same in the two cases, it follows from \eqref{lem9gen3}, \eqref{I1}, \eqref{I2} and \eqref{I3} that
\begin{equation}\label{I1'}
\begin{aligned}
\hat{I}_1\leq& \int^{t}_{0}(1+\tau)^{\hat{\mu}(\varepsilon)-\frac{3\gamma^2+3\gamma+1-\nu(3\gamma+2)}{(\gamma+1)^2}}
\big|\big|\rho-\bar{\rho}\big|\big|_{L^{\gamma+1}}d\tau\\
\leq& \int^{t}_{0}\int(1+\tau)^{\theta^*}\big|\rho-\bar{\rho}\big|^{\gamma+1}dxd\tau\\
&+C\int^{t}_{0}(1+\tau)^{\left[\hat{\mu}(\varepsilon)-\frac{3\gamma^2+3\gamma+1-\nu(3\gamma+2)}{(\gamma+1)^2}-\frac{\theta^*}{\gamma+1}\right]
\frac{\gamma+1}{\gamma}}d\tau\\
\leq &C,
\end{aligned}
\end{equation}
\begin{equation}\label{I2'}
\begin{aligned}
\hat{I}_2\leq& \int^{t}_{0}(1+\tau)^{\hat{\mu}(\varepsilon)-\frac{4\gamma+1-3\nu}{2(\gamma+1)}}
||y_{\tau}||d\tau\\
\leq& \int^{t}_{0}\int(1+\tau)^{\mu^*(\varepsilon)-\nu}y^{2}_{\tau}dxd\tau\\
&+C\int^{t}_{0}(1+\tau)^{2\left[\hat{\mu}(\varepsilon)-\frac{4\gamma+1-3\nu}{2(\gamma+1)}-\frac{\mu^*(\varepsilon)-\nu}{2}\right]}
d\tau\\
\leq &C,
\end{aligned}
\end{equation}
and
\begin{equation}\label{I3'}
\begin{aligned}
\hat{I}_3\leq& C\int^{t}_{0}(1+\tau)^{\hat{\mu}(\varepsilon)-1+\frac{1+\nu}{(\gamma+1)^2}}
\left(\int
\left(\rho^{\gamma}-\bar{\rho}^{\gamma}\right)\left(\rho-\bar{\rho}\right)  dx\right)^{\frac{\gamma}{\gamma+1}}d\tau\\
\leq& \int^{t}_{0}\int(1+\tau)^{\theta^*}
\left(\rho^{\gamma}-\bar{\rho}^{\gamma}\right)\left(\rho-\bar{\rho}\right)dxd\tau\\
&+C\int^{t}_{0}(1+\tau)^{\left[\hat\mu(\varepsilon)-1+\frac{1+\nu}{(\gamma+1)^2}-\frac{\gamma\theta^*}{\gamma+1}\right](\gamma+1)}
d\tau\\
\leq &C
\end{aligned}
\end{equation}
due to
\begin{align*}
&\hat{\mu}(\varepsilon)-\frac{3\gamma^2+3\gamma+1-\nu(3\gamma+2)}{(\gamma+1)^2}-\frac{\theta^*}{\gamma+1}+\frac{\gamma}{\gamma+1}\\
=&\frac{2\gamma-1-\nu}{\gamma+1}-\frac{\gamma}{\gamma+1}\varepsilon-\frac{2\gamma^2+3\gamma+1-\nu(3\gamma+3)}{(\gamma+1)^2}\\
=&\frac{2\nu-2}{\gamma+1}-\frac{\gamma}{\gamma+1}\varepsilon<0,
\end{align*}
\begin{align*}
&\hat{\mu}(\varepsilon)-\frac{4\gamma+1-3\nu}{2(\gamma+1)}-\frac{\mu^*(\varepsilon)-\nu}{2}+\frac{1}{2}\\
&=\frac{2\gamma-1-\nu}{\gamma+1}-\frac{\varepsilon}{2}-\frac{1}{4}+\frac{1}{2}\left(\frac{\nu}{2}+\frac{\nu+1}{\gamma+1}\right)
-\frac{4\gamma+1-3\nu}{2(\gamma+1)}\\
&=\frac{\nu-1}{\gamma+1}+\frac{\nu-1}{4}-\frac{\varepsilon}{2}<0,
\end{align*}
and
\begin{align*}
\hat{\mu}(\varepsilon)-1+\frac{1+\nu}{(\gamma+1)^2}-\frac{\gamma\theta^*}{\gamma+1}+\frac{1}{\gamma+1}=-\frac{\varepsilon}{\gamma+1}<0.
\end{align*}
by \eqref{1.55}, \eqref{1.56} and $\nu \in \left[\frac{\gamma}{\gamma +2} , 1\right)$.

As for $\hat{I}_4$, we also rewrite it as follows
\begin{equation}\label{I412'}
\begin{aligned}
\hat{I}_4&\leq C\int^{t}_{0}\int\hat{\mu}(\varepsilon)(1+\tau)^{\hat{\mu}(\varepsilon)-1}Q_*d xd\tau\\
&\leq C\int^{t}_{0}\int(1+\tau)^{\hat{\mu}(\varepsilon)-1}
\left[\frac{m^2}{\rho}+\frac{\bar{m}^2}{\bar{\rho}}
+\left|\frac{\bar{m}^2}{\bar{\rho}^2}(\rho-\bar{\rho})\right|+\left|\frac{2\bar{m}}{\bar{\rho}}(m-\bar{m})\right|\right]dxd\tau\\
&=:C\left(\hat{I}_{41}+\hat{I}_{42}+\hat{I}_{43}+\hat{I}_{44}\right).
\end{aligned}
\end{equation}
First, we estimate $\hat{I}_{41}$ and $\hat{I}_{42}$. It follows from $\nu\in\left[\frac{\gamma}{\gamma+2},1\right)$, \eqref{omeganu} and \eqref{omegavar} that
\begin{equation}\label{I41'}
\begin{aligned}
\hat{I}_{41}=\int^{t}_{0}\int(1+\tau)^{\hat{\mu}(\varepsilon)-1}\frac{m^2}{\rho}dxd\tau\leq C,
\end{aligned}
\end{equation}
where the last inequality is from
\begin{align*}
&\hat{\mu}(\varepsilon)-1-(\omega-\nu)=\frac{\nu-1}{\gamma+1}<0,
\end{align*}
and
\begin{equation}\label{I42'}
\begin{aligned}
\hat{I}_{42}=\int^{t}_{0}\int(1+\tau)^{\hat{\mu}(\varepsilon)-1}\frac{\bar{m}^2}{\bar{\rho}}dxd\tau
\leq\int^{t}_{0}(1+\tau)^{\hat{\mu}(\varepsilon)-1-\frac{2\gamma-2\nu}{\gamma+1}}d\tau\leq C,
\end{aligned}
\end{equation}
due to
\begin{align*}
\hat{\mu}(\varepsilon)-\frac{2\gamma-2\nu}{\gamma+1}=\frac{2\gamma-1-\nu}{\gamma+1}-\frac{2\gamma-2\nu}{\gamma+1}-\varepsilon
<\frac{\nu-1}{\gamma+1}<0.
\end{align*}
Now, we estimate $\hat{I}_{43}$ and $\hat{I}_{44}$. Similar to \eqref{I43} and \eqref{I44}, we deduce that
\begin{equation}\label{I43'}
\begin{aligned}
\hat{I}_{43}=&\int^{t}_{0}\int(1+\tau)^{\hat{\mu}(\varepsilon)-1}\left|\frac{\bar{m}^2}{\bar{\rho}^2}(\rho-\bar{\rho})\right|dxd\tau\\
\leq&\int^{t}_{0}(1+\tau)^{\hat{\mu}(\varepsilon)-1-\frac{2\gamma^2+\gamma-\nu(3\gamma+2)}
{(\gamma+1)^2}}\big|\big|\rho-\bar{\rho}\big|\big|_{L^{\gamma+1}}d\tau\\
\leq&\int^{t}_{0}(1+\tau)^{\left[\hat{\mu}(\varepsilon)-\frac{3\gamma^2+3\gamma+1-\nu(3\gamma+2)}{(\gamma+1)^2}-\frac{\theta^*}
{\gamma+1}\right]\frac{\gamma+1}{\gamma}}d\tau\\
&+C\int^{t}_{0}\int(1+\tau)^{\theta^*}\big|\rho-\bar{\rho}\big|^{\gamma+1}dxd\tau\\
\leq &C,
\end{aligned}
\end{equation}
and
\begin{equation}\label{I44'}
\begin{aligned}
\hat{I}_{44}=&\int^{t}_{0}\int(1+\tau)^{\hat{\mu}(\varepsilon)-1}\left|\frac{2\bar{m}}{\bar{\rho}}(m-\bar{m})\right|dxd\tau\\
\leq& C\int^{t}_{0}(1+\tau)^{\hat{\mu}(\varepsilon)-1-\frac{2\gamma-1-3\nu}{2(\gamma+1)}}\big|\big|y_\tau\big|\big|d\tau\\
\leq&\int^{t}_{0}(1+\tau)^{2\left[\hat{\mu}(\varepsilon)-\frac{4\gamma+1-3\nu}{2(\gamma+1)}-\frac{\mu^*(\varepsilon)-\nu}
{2}\right]}d\tau\\
&+C\int^{t}_{0}\int(1+\tau)^{\mu^*(\varepsilon)-\nu}y^{2}_{\tau}dxd\tau\\
\leq& C,
\end{aligned}
\end{equation}
by \eqref{lem9gen3}, \eqref{I1'} and \eqref{I2'}.

Substituting \eqref{I41'}-\eqref{I44'} into \eqref{I412'}, we have $\hat{I}_4\leq C$, which together with \eqref{etamuv'}-\eqref{I3'} yields
that
$$
(1+t)^{\hat{\mu}(\varepsilon)} \int \eta_* d x +\alpha\int^{t}_{0}\int(1+\tau)^{\hat{\mu}(\varepsilon)-\nu}Q_* dxd\tau\leq C,
$$
and so justifies \eqref{phivar} by \eqref{hatmuvar} for $ \nu \in\left[\frac{\gamma}{\gamma+2}, 1\right)$.
\end{proof}

We now prove Theorem~\ref{thm1}. From \eqref{eta_*}, \eqref{pQ} and Theorem \ref{thm4}, we have
$$
(1+t)^{\phi(\varepsilon)} \int P_* d x\leq C.
$$
Combining this, \eqref{1phi} and Lemma \ref{lem32}, we obtain \eqref{1rr+1} and \eqref{1.5} with $\gamma\ge2$.

As for $\eqref{1rr+1'}$ and \eqref{1.6}, we distinguish two cases.

\vspace{4pt}
\noindent\textbf{$Case~ A: \gamma\ge2.$}

In this case,
$$
P_*=\rho^{\gamma}-\bar{\rho}^{\gamma}-\gamma\bar{\rho}^{\gamma-1}(\rho-\bar{\rho}) \geq c_\gamma \bar{\rho}^{\gamma-2}|\rho-\bar{\rho}|^2, \quad \gamma \geq 2.
$$
Hence, by Lemma \ref{lem31.5}, the relative entropy estimate in Theorem \ref{thm4} and \eqref{rhobeta}, we obtain the $L^1$ convergence rate for the density:
\begin{align*}
\int_{-\infty}^{+\infty}|\rho-\bar{\rho}| d x &\leq 2 \int_{\Omega_t}|\rho-\bar{\rho}|\mathrm{d} x\\
&\leq 2\left(\int_{\Omega_t} \bar{\rho}^{\gamma-2}|\rho-\bar{\rho}|^2 d x\right)^{\frac{1}{2}}\left(\int_{\Omega_t} \bar{\rho}^{2-\gamma} d x\right)^{\frac{1}{2}}\\
&\leq C(1+t)^{-\frac{\phi(\varepsilon)}{2}}(1+t)^{\frac{\gamma-1}{2(\gamma+1)}(1+\nu)}.
\end{align*}
This and \eqref{1phi} together justify $\eqref{1rr+1'}$ for $\gamma\ge2$.

\vspace{4pt}
\noindent\textbf{$Case~ B: 1<\gamma<2.$}

In this case, let
\begin{align*}
\Omega_1&=\left\{x:\rho=0, \mbox{or}~ \,\bar{\rho}=0, \mbox{or}~\,  |\rho-\bar{\rho}|\ge\frac{1}{2}\bar{\rho}\right\},\\
\Omega_2&=\left\{x:\rho\neq0, \mbox{and}~\,\bar{\rho}\neq0, \mbox{and}~\, |\rho-\bar{\rho}|<\frac{1}{2}\bar{\rho}\right\}.
\end{align*}
Then, by Lemma \ref{lem31.5}, we have
\begin{equation}\label{L1}
\begin{aligned}
\int^{+\infty}_{-\infty}|\rho-\bar{\rho}| d x&\leq 2\int_{\{x:\bar{\rho}>0\}}|\rho-\bar{\rho}| d x\\
&\leq 2\int_{\Omega_1\cap\{x:\bar{\rho}>0\}}|\rho-\bar{\rho}| d x+2\int_{\Omega_2\cap\{x:\bar{\rho}>0\}}|\rho-\bar{\rho}| d x.
\end{aligned}
\end{equation}
On $\Omega_1\cap\Omega_t$, from Lemmas \ref{lem33} and \ref{lem4}, \eqref{rho0}, \eqref{eta_*}, Theorem \ref{thm4} and the H\"older inequality, we see that for any $0<\delta<\frac{(\gamma-1)^2}{2\gamma}$,
\begin{equation}\label{omega1}
\begin{aligned}
\int_{\Omega_1\cap\{x:\bar{\rho}>0\}}|\rho-\bar{\rho}| d x
&\leq\int_{\Omega_1\cap\{x:\bar{\rho}>0\}}\bar{\rho}^{-\delta}|\rho-\bar{\rho}|\bar{\rho}^{\delta} d x\\
&\leq\left(\int_{\Omega_1}|\rho-\bar{\rho}|^{\gamma} d x\right)^{\frac{1}{\gamma}}\big|\big|\bar{\rho}^{\delta}\big|\big| _{L^{\frac{2\gamma}{\gamma-1}}(\Omega_t)}\big|\big|\bar{\rho}^{-\delta}\big|\big| _{L^{\frac{2\gamma}{\gamma-1}}(\Omega_t)}\\
&\leq C\left(\int_{\Omega_1}\rho^{\gamma}-\bar{\rho}^{\gamma}-\gamma\bar{\rho}^{\gamma-1}(\rho-\bar{\rho}) d x\right)^{\frac{1}{\gamma}}
\big|\big|\bar{\rho}^{\delta}\big|\big| _{L^{\frac{2\gamma}{\gamma-1}}(\Omega_t)}\big|\big|\bar{\rho}^{-\delta}\big|\big| _{L^{\frac{2\gamma}{\gamma-1}}(\Omega_t)}\\
&\leq C\left(\int\eta_* d x\right)^{\frac{1}{\gamma}}
(1+t)^{\frac{\gamma-1}{\gamma(\gamma+1)}(1+\nu)}\\
&\leq C(1+t)^{-\frac{\phi(\varepsilon)}{\gamma}}(1+t)^{\frac{\gamma-1}{\gamma(\gamma+1)}(1+\nu)}.
\end{aligned}
\end{equation}
On the other hand, on $\Omega_2\cap\Omega_t$, from Lemmas \ref{lem33} and \ref{lem4}, \eqref{rho01/2}, \eqref{eta_*}, Theorem \ref{thm4} and the H\"older inequality, we deduce that
\begin{equation}\label{omega2}
\begin{aligned}
\int_{\Omega_2\cap\{x:\bar{\rho}>0\}}|\rho-\bar{\rho}| d x&\leq\left(\int_{\Omega_2\cap\{x:\bar{\rho}>0\}} \bar{\rho}^{\gamma-2}|\rho-\bar{\rho}|^2 d x\right)^{\frac{1}{2}}\left(\int_{\Omega_2\cap\{x:\bar{\rho}>0\}}\bar{\rho}^{2-\gamma} d x\right)^{\frac{1}{2}}\\
&\leq\left(\int_{\Omega_2} \bar{\rho}^{\gamma-2}|\rho-\bar{\rho}|^2 d x\right)^{\frac{1}{2}}\left(\int_{\{x:\bar{\rho}>0\}}\bar{\rho}^{2-\gamma} d x\right)^{\frac{1}{2}}\\
&\leq C\left(\int_{\Omega_2}\rho^{\gamma}-\bar{\rho}^{\gamma}-\gamma\bar{\rho}^{\gamma-1}(\rho-\bar{\rho}) d x\right)^{\frac{1}{2}}\left(\int_{\{x:\bar{\rho}>0\}}\bar{\rho}^{2-\gamma} d x\right)^{\frac{1}{2}}\\
&\leq C\left(\int \eta_* d x\right)^{\frac{1}{2}}(1+t)^{\frac{\gamma-1}{2(\gamma+1)}(1+\nu)}\\
&\leq C(1+t)^{-\frac{\phi(\varepsilon)}{2}}(1+t)^{\frac{\gamma-1}{2(\gamma+1)}(1+\nu)}.
\end{aligned}
\end{equation}
Using \eqref{L1}, \eqref{omega1} and \eqref{omega2} yields
\begin{align*}
\int|\rho-\bar{\rho}| d x& \leq2\int_{\Omega_1\cap\{x:\bar{\rho}>0\}}|\rho-\bar{\rho}| d x
+2\int_{\Omega_2\cap\{x:\bar{\rho}>0\}}|\rho-\bar{\rho}| d x \\
&\leq C(1+t)^{-\frac{1}{\gamma}\left[\phi(\varepsilon)-\frac{\gamma-1}{\gamma+1}(1+\nu)\right]}+
C(1+t)^{-\frac{1}{2}\left[\phi(\varepsilon)-\frac{\gamma-1}{\gamma+1}(1+\nu)\right]}\\
&\leq C(1+t)^{-\frac{1}{2}\left[\phi(\varepsilon)-\frac{\gamma-1}{\gamma+1}(1+\nu)\right]},
\end{align*}
where the last inequality is from $1<\gamma<2$ and $\phi(\varepsilon)-\frac{\gamma-1}{\gamma+1}(1+\nu)>0$. Thus, combining the above inequality and \eqref{1phi}, we obtain \eqref{1rr+1'} with $1<\gamma<2$.

This completes the proof of Theorem \ref{thm1}.

\vspace{0.4cm}

{\bf Declarations}

\vspace{0.4cm}

{\bf Data availability statement: } No datasets were generated or analyzed in this study.

\vspace{0.4cm}

{\bf Conflict of interest:} The authors declare no conflict of interest.


\footnotesize
\begin{thebibliography}{99}
\bibitem{CAR} J. Carrillo, Y.P. Choi, O. Tse, Convergence to equilibrium in Wasserstein distance for damped Euler equations with interaction forces, Comm. Math. Phys. 365 (2019) 329-361.
\bibitem{CHE1} G.Q. Chen, H. Frid, Divergence-measure fields and hyperbolic conservation laws, Arch. Ration. Mech. Anal. 147 (1999) 89-118.
\bibitem{CUI1} H.B. Cui, H.Y. Yin, J.S. Zhang, C.J. Zhu, Convergence to nonlinear diffusion waves for solutions of Euler equations
    with time-depending damping, J. Differential Equations 264 (2018) 4564-4602.
\bibitem{CUI} H.B. Cui, H.Y. Yin, C.J. Zhu, L.M. Zhu, Convergence to diffusion waves for solutions of Euler equations with time-depending damping on quadrant, Sci. China Math. 62 (2019) 33-62.
\bibitem{CUI2} H.B. Cui, H.Y. Yin, C.J. Zhu, Convergence rates to the Barenblatt solutions for the compressible Euler equations with time-dependent damping, J. Differential Equations 374 (2023) 761-788.
\bibitem{EIT} T. Eiter, S. Schindler, Time-asymptotic self-similarity of the damped compressible Euler equations in parabolic scaling variables, Nonlinear Anal. 271 (2026) 114124.
\bibitem{GEN1} S.F. Geng, F.M. Huang, $L^1$-convergence rates to the Barenblatt solution for the damped compressible Euler equations, J. Differential Equations 266 (2019) 7890-7908.
\bibitem{GEN3} S.F. Geng, F.M. Huang, X.C. Wu, $L^1$-convergence to generalized Barenblatt solution for compressible Euler equations with time-dependent damping, SIAM J. Math. Anal. 53 (2021) 6048-6072.
\bibitem{GUT} T.S. Gutleb, J.A. Carrillo, S. Olver, Computation of power law equilibrium measures on balls
of arbitrary dimension, Constr. Approx. 58 (2023) 75--120.
\bibitem{HSI} L. Hsiao, T.P. Liu, Convergence to nonlinear diffusion waves for solutions of a system of hyperbolic conservation laws with damping, Commun. Math. Phys. 143 (1992) 599-605.
\bibitem{HOU} F. Hou, I. Witt, H.C. Yin, Global existence and blowup of smooth solutions of 3-D potential equations with time-dependent damping, Pac. J. Math. 292 (2018) 389-426.
\bibitem{HOU1} F. Hou, H.C. Yin, On the global existence and blowup of smooth solutions to the multi-dimensional compressible Euler equations with time-depending damping, Nonlinearity 30 (2017) 2485-2517.
\bibitem{HUA2} F.M. Huang, P. Marcati, R.H. Pan, Convergence to the Barenblatt solution for the compressible Euler equations with
    damping and vacuum, Arch. Ration. Mech. Anal. 176 (2005) 1-24.
\bibitem{HUA4} F.M. Huang, R.H. Pan, Z. Wang, $L^1$ convergence to the Barenblatt solution for compressible Euler equations with
    damping, Arch. Ration. Mech. Anal. 200 (2011) 665-689.
\bibitem{JIA} M.N. Jiang, C.J. Zhu, Convergence to strong nonlinear diffusion waves for solutions to p-system with damping on
    quadrant, J. Differential Equations 246 (2009) 50-77.
\bibitem{LI} H.T. Li, J.Y. Li, M. Mei, K.J. Zhang, Convergence to nonlinear diffusion waves for solutions of p-system with
    time-dependent damping, J. Math. Anal. Appl. 456 (2017) 849-871.
\bibitem{LIO} P.L. Lions, B. Perthame, E. Tadmor, Kinetic formulation of the isentropic gas dynamics and p-systems, Commun. Math. Phys. 163 (1994) 415-431.
\bibitem{LIU1} T.P. Liu, Compressible flow with damping and vacuum, Jpn. J. Ind. Appl. Math. 13 (1996) 25-32.
\bibitem{NIS1} K. Nishihara, Convergence rates to nonlinear diffusion waves for solutions of system of hyperbolic conservation
    laws with damping, J. Differential Equations 131 (1996) 171-188.
\bibitem{PAN1} X.H. Pan, Global existence of solutions to 1-d Euler equations with time-dependent damping, Nonlinear Anal. 132
    (2016) 327-336.
\bibitem{PAN2} X.H. Pan, Blow up of solutions to 1-d Euler equations with time-dependent damping, J. Math. Anal. Appl. 442
    (2016) 435-445.
\bibitem{SID} T.C. Sideris, B. Thomases, D.H. Wang, Long time behavior of solutions to the 3D compressible Euler equations
    with damping, Commun. Partial Differ. Equ. 28 (2003) 795-816.
\bibitem{SUG} Y. Sugiyama, Singularity formation for the 1D compressible Euler equations with variable damping coefficient,
    Nonlinear Anal. 170 (2018) 70-87.
\bibitem{SUG1} Y. Sugiyama, Remark on global existence of solutions to the 1D compressible Euler equation with time-dependent damping, Adv. Stud. Pure Math. 85 (2020) 379-389.
\bibitem{ZEN} H.H. Zeng, Time-asymptotics of physical vacuum free boundaries for compressible inviscid flows with damping, Calc. Var. Partial Differential Equations 61 (2022) 59.
\bibitem{ZEN1} H.H. Zeng, Almost global solutions to the three-dimensional isentropic inviscid flows with damping in a physical vacuum around Barenlatt solutions, Arch. Ration. Mech. Anal. 239 (2021) 553-597.
\bibitem{ZHA} H.J. Zhao, Convergence to strong nonlinear diffusion waves for solutions of p-system with damping, J. Differential Equations
    174 (2001) 200-236.
\bibitem{ZHU1} C.J. Zhu, Convergence rates to nonlinear diffusion waves for weak entropy solutions to p-system with damping,
    Sci. China Ser. A 46 (2003) 562-575.



\end{thebibliography}
\end{document}